\documentclass[10pt,reqno]{amsart}

\usepackage[latin2]{inputenc}
\usepackage{amsmath}
\usepackage{amsrefs}
\usepackage{graphicx}
\usepackage{amssymb}
\usepackage{esint}
\usepackage{color}
\usepackage{amsthm}
\usepackage{epsfig}
\usepackage{enumitem}
\usepackage{mathtools}
\usepackage{esvect}
\usepackage{tikz}
\usetikzlibrary{calc,intersections,through,backgrounds,patterns,arrows.meta, math,through}
\usepackage[labelformat = brace, position=top]{subcaption}
\usepackage[english]{babel}

\newtheorem{theorem}{Theorem}
\newtheorem{proposition}[theorem]{Proposition}
\newtheorem{lemma}[theorem]{Lemma}
\newtheorem{corollary}[theorem]{Corollary}
\newtheorem*{theorem*}{Theorem}

\theoremstyle{definition}
\newtheorem{definition}[theorem]{Definition}
\newtheorem{remark}[theorem]{Remark}

\def\XXint#1#2#3{{\setbox0=\hbox{$#1{#2#3}{\int}$ }
\vcenter{\hbox{$#2#3$ }}\kern-.6\wd0}}

\definecolor{Yellow}{rgb}{0.95,0.9,0.0} 
\definecolor{Red}{rgb}{0.8,0.1,0.1}
\definecolor{Green}{rgb}{0.1,0.65,0.2}
\definecolor{Blue}{rgb}{0.1,0.1,0.8}
\definecolor{Purple}{rgb}{0.7,0.1,0.7}
\definecolor{Grey}{rgb}{0.6,0.6,0.6}
\definecolor{LightRed}{rgb}{0.8,0.5,0.5}
\definecolor{LightBlue}{rgb}{0.3,0.2,0.7}

\newcommand{\supp}{\operatorname{supp}}

\newcommand{\dist}{\operatorname{dist}}

\newcommand {\jump}[1] {[\![ #1 ]\!]}

\newcommand{\Rd}[1][d]{{\mathbb{R}^{#1}}}
\allowdisplaybreaks[4]

\newcommand{\no}{\mathbf{n}}

\newcommand{\Tweak}{T_{0}}
\newcommand{\Tstrong}{T_{*}}
\newcommand{\eps}{\varepsilon}

\usepackage[linktocpage=true,colorlinks=true,linkcolor=Blue,citecolor=Green]{hyperref}

\usepackage{mathrsfs,soul}

\renewcommand{\vv}{\mathbf{v}}
 \newcommand{\II}{\mathcal{I}}
\renewcommand{\O}{\Omega}
\renewcommand{\(}{\left(}
\renewcommand{\)}{\right)}

\renewcommand{\div}{\operatorname{div}}

\begin{document}

\title[Sharp interface limit for a Navier--Stokes/Allen--Cahn system]
{The sharp interface limit of a Navier--Stokes/Allen--Cahn system
with constant mobility: Convergence rates by a relative energy approach}

\author{Sebastian Hensel}
\address{Hausdorff Center for Mathematics, Universit{\"a}t Bonn, Endenicher Allee 62, 53115 Bonn, Germany}
\email{sebastian.hensel@hcm.uni-bonn.de}

\author{Yuning Liu}
\address{NYU Shanghai, 1555 Century Avenue, Shanghai 200122, China, and NYU-ECNU Institute of
Mathematical Sciences at NYU Shanghai, 3663 Zhongshan Road North, Shanghai, 200062, China}
\email{yl67@nyu.edu}


\begin{abstract}
We investigate the sharp interface limit of a diffuse interface system that  
couples the Allen--Cahn equation with the instationary Navier--Stokes system 
in a bounded domain in $\mathbb{R}^d$ with $d \in \{2,3\}$. This model is used to  
describe a propagating front in a viscous incompressible flow with the width of 
the transition layer being characterized by a small  parameter $\eps>0$. 
We show that the solutions converge to a limit two-phase fluid system with surface tension that
couples the mean curvature flow and the Navier--Stokes system. The main assumptions
are that the evolution of the limit system is sufficiently regular and that
the associated evolving interface does not intersect the boundary of the container.
For quantitatively well-prepared initial data, we even establish an optimal convergence rate.
This is the first rigorous result of this kind which is valid in all
physically relevant ambient dimensions.


\end{abstract}

\maketitle
\tableofcontents

\section{Introduction}

\subsection{Context}
Curvature driven interface evolution is a challenging topic in PDE theory
due to the inherent emergence of topology changes. Indeed, classical
descriptions (i.e., via parametrization of the evolving interface) cease
to work once the solution approaches the first time of such a topology change. 
One popular tool to describe the dynamics even after these consists of
phase-field models, where instead of a sharp interface one considers
a diffuse interface layer of finite width (which is typically small related to
a given scaling parameter~$\eps$). In the specific example of the evolution
of two (macroscopically) immiscible fluids in the presence of surface tension, 
an associated fundamental phase-field model is the so-called \textit{model~H} 
coupling Navier--Stokes dynamics with the fourth-order Cahn--Hilliard dynamics.
This model was in fact introduced in~\cites{GURTIN1996,HohenbergHalperin}.

Next to obvious PDE~questions concerning existence and uniqueness of solutions to 
such a phase-field model, another highly relevant and natural question
is the consistency with a sharp interface model. Indeed, by formally taking the limit
$\eps \to 0$ one hopes to identify the underlying dynamics
in terms of the sharp interface model under consideration, and therefore
to justify the transition to the phase-field model. Needless to say,
for numerical applications apart from qualitative convergence results 
also quantitative results in the form of convergence rates in certain norms 
are of interest.

In the context of model~H (and even for more general models
allowing, e.g., a non-constant density), the formal 
derivation of the associated sharp interface limit model by means 
of asymptotic expansion techniques is performed in~\cite{AbelsGarckeGruen}. 
Rigorous justifications of these arguments are scarce though. For global-in-time
convergence of solutions to model~H (or a related model) to a rather weak
solution to the sharp interface limit model (i.e., a concept
of varifold solutions), one may consult~\cite{AbelsLengeler}
(or~\cite{AbelsRoeger}, respectively). Short-time (quantitative) 
convergence to strong solutions of the sharp interface
limit model is in turn shown in the recent works~\cite{Abels2021a} and~\cite{Abels2021b},
where, however, only the setting of the stationary Stokes operator is treated.
At the time of this writing, the rigorous derivation of convergence rates
incorporating the full Navier--Stokes dynamics in model~H remains an 
important open problem in the field.

In the present work, we study instead of model~H (or related models)
a phase-field model with fluid mechanical coupling in ambient dimension~$d \in \{2,3\}$ 
which is based on the second-order Allen--Cahn operator. 
We further consider the scaling regime 
of constant mobility, and with respect to the fluid mechanical modeling we 
restrict ourselves to the case of Navier--Stokes dynamics with constant
density and constant viscosity (we refer to Subsection~\ref{subsec:models} 
for a mathematical formulation of the model). Phase-field models of such
Navier--Stokes/Allen--Cahn type were first introduced in~\cites{LiuShen,Jiang2017}.
Formal asymptotic expansions (cf.\ \cite{Abels2022a}) suggest that the associated sharp-interface 
limit model is given by a Navier--Stokes two-phase flow with surface tension,
where, as a consequence of the constant mobility assumption, the interface
separating the two fluid phases is not merely transported by the fluids
but is also subject to mean curvature flow (we again refer to Subsection~\ref{subsec:models} 
for a mathematical formulation of the corresponding model). In particular, the mass of each individual
fluid phase is not preserved and the resulting model may be interpreted as
a simplified model for two-phase fluid flow incorporating phase transitions and surface tension. 
For global-in-time existence of weak solutions to the phase-field model,
we refer to~\cite{MR1329830} (see also \cite{Gal2010} for longtime behavior of solutions).

In this setting, the main result of the present work rigorously justifies the formally
derived sharp interface limit model in ambient dimension~$d \in \{2,3\}$.
Our convergence result holds true on the time horizon of existence
of a (sufficiently regular) strong solution of the sharp interface limit.
We even derive sharp convergence rates in strong norms under the
assumption of correspondingly well-prepared initial data.
For precise mathematical statements of these two main results, we refer
to Theorem~\ref{theo:mainResult} and Corollary~\ref{cor:sharpConvergenceRates} below.
We mention that a similar convergence result was recently established 
in~\cite{Abels2022} (cf.\ also \cite{Jiang2022})
by a completely different approach (we provide further comments on the methods later).
The results of~\cite{Abels2022} are, however, restricted to ambient dimension~$d=2$,
so that our work, to the best of our knowledge, is the first to rigorously justify 
the above sharp interface limit in all physically relevant dimensions. It has to be said, 
though, that the authors of~\cite{Abels2022} are in addition able to treat the regime
of different viscosities. Even though we expect this to be manageable also in the 
framework of our strategy (without restrictions on the ambient dimension), 
this may very well lead to a doubling of the length of the present paper
(cf.\ the corresponding challenges in~\cite{Fischer2020c}). For this reason, we 
restrict ourselves to the most basic setting. We finally mention
that~\cite{AbelsLiu} contains a preliminary convergence result
preceding ours and the one of~\cite{Abels2022}, where the authors replace
the full Navier--Stokes dynamics by the stationary Stokes operator
(and again only in ambient dimension~$d=2$).

Without coupling to a fluid mechanical system, the rigorous
convergence of phase-field models to sharp interface evolutions
is of course an already extensively studied subject in the literature.
For (qualitative and/or quantitative) convergence of solutions
of the Allen--Cahn equation towards various weak or strong notions of solutions
for mean curvature flow, one may consult, e.g., the 
works~\cites{Evans1992,DeMottoni1995,ilmanen,mugnai-roeger,serfaty,Fischer2020b,Hensel2021l}.
For the inclusion of constant contact angles, corresponding results can
be found in~\cites{Owen1992,Katsoulakis1995,Barles1998,Lio2003,Mizuno2015,
Kagaya2018a,Abels2019,Abels2021,Hensel2021d,Hensel2021c},
whereas (qualitative and/or quantitative) convergence of solutions
of the vectorial Allen--Cahn equation towards evolution by multiphase
mean curvature flow is the subject of~\cites{LauxSimon18,Fischer2022}.
The results on the connection of the Allen--Cahn equation with
mean curvature flow have a fourth-order analogue, namely,
the convergence of solutions of the Cahn--Hilliard equation
towards evolution by Mullins--Sekerka flow as, e.g., shown
in~\cites{Chen,Alikakos1994,Le2008} (cf.\ also the discussion in~\cite{serfaty}).
For a corresponding result with disparate mobilities, we refer
to the recent work~\cite{Kroemer2021}. Finally, we mention~\cites{Laux2021}
for a scaling limit result modeling nematic-isotropic phase
transitions in the context of Landau--De~Gennes theory of liquid crystals.

In order to establish the above convergence results, 
a variety of techniques and frameworks is used depending
on the precise goals (i.e., qualitative long-time convergence
towards weak solutions of the sharp interface limit model
vs.\ quantitative convergence towards sufficiently
smooth solutions of the sharp interface limit model
until the latter run into their first topology change).
In fact, most of the above works are either based on
notions from geometric measure theory (e.g., varifolds
and their first variations), gradient flow techniques
in the spirit of~\cite{Sandier2004} (cf.\ also \cite{serfaty}),
or combining rigorous asymptotic expansions with a
stability analysis of the linearized operator associated
with the phase-field model. Especially in the context
of the previously mentioned quantitative convergence results
incorporating a fluid mechanical coupling, the latter strategy seems
to be the only one used so far to the best of our knowledge.

However, in the recent inspiring work~\cite{Fischer2020b},
a new approach for the derivation of convergence rates 
was introduced and implemented for the simplest
setting of the Allen--Cahn equation posed on the full space~$\mathbb{R}^d$.
This strategy is closest to the gradient flow perspective
in the sense that it first generates a distance measure
for the difference of the phase-field and sharp interface
solutions based on the phase-field energy, and then estimates
its time evolution by a Gronwall-type argument based, amongst other things,
on the dissipation structure of the phase-field model.
The method developed in~\cite{Fischer2020b} is therefore
reminiscent of a well-established technique to establish
(weak-strong) uniqueness of solutions in the context of a variety of classical
continuum mechanics models (e.g., incompressible and compressible Navier--Stokes flow,
or conservation laws): the so-called relative entropy method.
In fact, the work~\cite{Fischer2020b} draws motivation from recent
results extending the relative entropy method (and thus weak-strong uniqueness) 
to problems incorporating geometric evolution. This was first implemented
for binormal curvature flow of curves in~$\mathbb{R}^3$ in~\cite{JerrardSmets},
or for Navier--Stokes two-phase flow with surface tension in~\cite{Fischer2020c}.
Subsequent extensions are able to deal with multiphase mean curvature flow~\cite{Fischer2020a}
(cf.\ also~\cite{Hensel2021}), contact angle problems~\cites{Hensel2021m,Hensel2021d},
or a novel notion of varifold solutions for mean curvature flow 
in the spirit of ideas from De~Giorgi~\cite{Hensel2021l}. 
In view of these developments, it is not surprising that also~\cite{Fischer2020b}
already led to several follow-up works in the context of scaling limits for phase-field models,
see~\cites{Laux2021,Hensel2021c,Liu2021,Fischer2022}.

In the present work, we continue this story and extend the approach from~\cite{Fischer2020b}
to the case of a Navier--Stokes/Allen--Cahn model with constant mobility.
We refer the reader to Section~\ref{sec:mainResult} and 
especially Section~\ref{sec:errorFunctionals} for a mathematical account 
on our strategy.

\subsection{The phase-field model and its sharp interface limit}
\label{subsec:models}
Let $\Omega \subset \Rd$, $d \in \{2,3\}$, be a bounded domain with orientable
and $C^2$-boundary. We then consider the following most basic
Navier--Stokes/Allen--Cahn problem in $\Omega$
\begin{subequations}
\begin{align}
\label{eq:phaseFieldModel1}
\partial_t v_\eps + (v_\eps\cdot\nabla)v_\eps &= 
\Delta v_\eps - \nabla \pi_\eps 
- \nabla\cdot\big(\eps\nabla\varphi_\eps\otimes\nabla\varphi_\eps\big) 
&&\text{in }\Omega\times (0,\Tweak),
\\
\label{eq:phaseFieldModel2}
\nabla\cdot v_\eps &= 0
&&\text{in }\Omega\times (0,\Tweak),
\\
\label{eq:phaseFieldModel3}
\partial_t\varphi_\eps + (v_\eps\cdot\nabla)\varphi_\eps 
&= \Delta\varphi_\eps - \smash{\frac{1}{\eps^2}W'(\varphi_\eps)} 
&&\text{in }\Omega\times (0,\Tweak),
\\
\label{eq:phaseFieldModel4}
v_\eps(\cdot,0) 
&= v_{\eps,0} 
&&\text{in }\Omega,
\\
\label{eq:phaseFieldModel5}
\varphi_\eps(\cdot,0) 
&= \varphi_{\eps,0} \in [-1,1]
&&\text{in }\Omega,
\\
\label{eq:phaseFieldModel6}
(\varphi_\eps,v_\eps) &= (-1,0)
&&\text{on } \partial\Omega\times (0,\Tweak).
\end{align}
\end{subequations}	
Here, $W$ denotes a double-well potential satisfying standard assumptions
(see, e.g.,~\cite{AbelsLiu}). We fix
the associated surface tension constant by 
\begin{align*}
c_0:=\int_{-1}^{1} \sqrt{2W(r)}\,dr.
\end{align*}
In the sharp interface limit $\eps\to 0$ one expects to obtain 
a simplified model for a two-phase fluid flow with phase transitions,
which in our case reads 
\begin{subequations} 
\begin{align}
\label{eq:sharpInterfaceLimit1}
\partial_t v + (v\cdot\nabla)v &= \Delta v - \nabla \pi 
&&\text{in } \big(\Omega {\times} (0,\Tstrong)\big) \setminus \II,
\\
\label{eq:sharpInterfaceLimit2}
\nabla\cdot v &= 0
&&\text{in }\Omega {\times} (0,\Tstrong),
\\
\label{eq:sharpInterfaceLimit3}
\partial_t\chi + (v\cdot\nabla)\chi &= -H_{\II}|\nabla\chi| 
&&\text{in }\Omega {\times} (0,\Tstrong),
\\
\label{eq:sharpInterfaceLimit4}
\jump{v} &= 0,
&&\text{on } \II,
\\
\label{eq:sharpInterfaceLimit5}
\jump{2\nabla^\mathrm{sym}v - p\mathrm{Id}}\frac{\nabla\chi}{|\nabla\chi|}
&= c_0 H_{\II}\frac{\nabla\chi}{|\nabla\chi|}
&&\text{on } \II,
\\
\label{eq:sharpInterfaceLimit6}
v(\cdot,0) &= v_0
&&\text{in }\Omega,
\\
\label{eq:sharpInterfaceLimit7}
\chi(\cdot,0) &= \chi_0
&&\text{in }\Omega,
\\
\label{eq:sharpInterfaceLimit8}
(\chi,v) &= (0,0)
&&\text{on }\partial\Omega {\times} (0,\Tstrong).
\end{align}
\end{subequations}
In the above system, $\chi$ denotes a time-dependent characteristic function.
From now on, we assume that we are provided with a sufficiently regular solution $(\chi,v)$ for
the sharp interface limit model (for precise assumptions, we refer to 
Definition~\ref{def:strongSolSharpInterfaceLimit} below). 
In particular, $\supp|\nabla\chi| \cap (\Omega {\times} [0,\Tstrong))$ 
shall model a sufficiently regular interface 
\begin{equation}
\II={\bigcup_{t \in [0,\Tstrong)}
\II(t) {\times} \{t\}}
\end{equation}
 in space-time associated with an open and sufficiently regular 
set~$$\Omega_+ = {\bigcup_{t \in [0,\Tstrong)}
\Omega_+(t) {\times} \{t\}} \subset \Omega {\times} (0,\Tstrong).$$
The map $\chi$ is assumed to be $\equiv 1$ in~$\Omega_+$
as well as $\equiv 0$ in $(\Omega{\times}[0,T_*))\setminus\overline{\Omega_+}$. 
Denoting for all $t \in (0,\Tstrong)$ by $\no_{\II}(\cdot,t)$ the associated unit normal
along~$\II(t)$ pointing inside~$\Omega_+(t)$, the evolution equation 
for the interface then translates into (all scalar geometric
quantities are oriented with respect to the normal vector field~$\no_{\II}$)
\begin{align}
\label{eq:geomEvolutionSharpInterface}
V_{\II} &= \no_{\II}\cdot v + H_{\II}
&\text{on } \II.
\end{align}  
Here, $V_{\II}(\cdot,t)$ and $H_{\II}(\cdot,t)$ denote the normal speed and
the scalar mean curvature of the interface~$\II(t)$, $t \in [0,\Tstrong)$.
Finally, in order to avoid issues originating from contact point dynamics, we assume
throughout the rest of this work that
\begin{align}
\label{eq:noContactPoints}
\overline{\Omega_+(t)} &\subset \Omega 
&&\text{for all } t \in [0,\Tstrong).
\end{align}

\section{Main result}
\label{sec:mainResult}

The main result of the present work is concerned with
a rigorous convergence result for solutions
of the phase-field model~\eqref{eq:phaseFieldModel1}--\eqref{eq:phaseFieldModel6}
(even quantitatively with sharp convergence rates when starting from well-prepared initial data) 
towards strong solutions of the sharp interface limit
model~\eqref{eq:sharpInterfaceLimit1}--\eqref{eq:sharpInterfaceLimit8}.
As already mentioned in the introduction, our approach is not based
on the strategy of combining rigorous asymptotic expansions 
with stability estimates for the underlying linearized operator.
In contrast, our results are facilitated by the introduction of two
error functionals which aim to encode the difference between 
a solution of the phase-field model and a solution for the expected
sharp interface limit. More precisely, directly inspired by the recent
work of Fischer, Laux and Simon~\cite{Fischer2020b} on the (quantitative) convergence
of solutions of the Allen--Cahn equation to classical solutions
of mean curvature flow, we will work on $[0,T_*)$ with a relative energy
\begin{equation}
\begin{aligned}
\label{eq:relEnergyResults}
E[\varphi_\eps,v_\eps|\chi,v]
&:= \int_{\Omega} \frac{1}{2}\big|v_\eps {-} v\big|^2 \,dx 
\\&~~~~
+ \int_{\Omega} \frac{\eps}{2}\big|\nabla\varphi_\eps\big|^2 + 
\frac{1}{\eps}W\big(\varphi_\eps\big) 
- \xi\cdot \nabla \Big(\int_{-1}^{\varphi_\eps} \sqrt{2W(r)}\,dr\Big) \,dx
\end{aligned}
\end{equation}
as well as with a suitable measure for the difference in the phase indicators
\begin{align}
\label{eq:volErrorResults}
E_{\mathrm{vol}}[\varphi_\eps|\chi] := 
\int_{\Omega} \bigg| c_0\chi - 
\int_{-1}^{\varphi_\eps} \sqrt{2W(r)} \,dr \bigg| |\vartheta| \,dx.
\end{align}
In the above definitions, the vector field~$\xi$
will be a suitable extension of the unit normal vector field
of the smoothly evolving interface~$\II$ whereas~$\vartheta$
will be a suitable truncation of the associated signed distance function
(for precise assumptions on~$\xi$ and~$\vartheta$, we refer
to Subsection~\ref{subsec:relEnergy} and Subsection~\ref{subsec:volError},
respectively). Both qualitative and quantitative convergence then follows
from a Gronwall-type stability estimate for the two error 
functionals~\eqref{eq:relEnergyResults} and~\eqref{eq:volErrorResults}.
The precise statement reads as follows.

\begin{theorem}[Error estimates for general initial data]
\label{theo:mainResult}
Let $d \in \{2,3\}$, let $\Omega\subset\Rd$
be a bounded domain with orientable and 
$C^2$-boundary~$\partial\Omega$,
and consider two finite time horizons $0 < \Tstrong \leq \Tweak < \infty$.
Let~$\eps \in (0,1)$ and let~$(\varphi_\eps,v_\eps)$ be a weak
solution of the Navier--Stokes/Allen--Cahn 
system~\emph{\eqref{eq:phaseFieldModel1}--\eqref{eq:phaseFieldModel6}}
with time horizon~$\Tweak$ and initial data~$(\varphi_{\eps,0},v_{\eps,0})$
in the sense of Definition~\ref{def:weakSolPhaseFieldModel}.
Finally, let~$(\chi,v)$ be a strong solution for the sharp interface
limit model~\emph{\eqref{eq:sharpInterfaceLimit1}--\eqref{eq:sharpInterfaceLimit8}}
with time horizon~$\Tstrong$ and initial data~$(\chi_0,v_0)$
in the sense of Definition~\ref{def:strongSolSharpInterfaceLimit}.

Then, for a.e.\ $T \in (0,\Tstrong)$ there exists a constant 
$C=C(\chi,v,T) \in (0,\infty)$, a continuous vector field
$\xi {=} \xi(\chi,v,T)\colon\Omega{\times}[0,T]\to\Rd$, and a continuous
weight $\vartheta {=} \vartheta(\chi,v,T)\colon\Omega{\times}[0,T]\to\mathbb{R}$
such that the relative energy $E[\varphi_\eps,v_\eps|\chi,v]$ 
and the error in the phase indicators~$E_{\mathrm{vol}}[\varphi_\eps|\chi]$
defined by~\eqref{eq:relEnergyResults} and~\eqref{eq:volErrorResults}, 
respectively, satisfy the estimates
\begin{align}
\label{eq:gronwallEstimateRelativeEntropy}
E[\varphi_\eps,v_\eps|\chi,v](t) &\leq e^{Ct} 
\big(E[\varphi_\eps,v_\eps|\chi,v](0) + 
E_{\mathrm{vol}}[\varphi_\eps|\chi](0)\big),
&&\text{for a.e.\ } t \in [0,T],
\\ 
\label{eq:gronwallEstimatePhaseError}
E_{\mathrm{vol}}[\varphi_\eps|\chi](t) &\leq 
e^{Ct} \big(E[\varphi_\eps,v_\eps|\chi,v](0) + 
E_{\mathrm{vol}}[\varphi_\eps|\chi](0)\big),
&&\text{for a.e.\ } t \in [0,T].
\end{align}
\end{theorem}
Based on the straightforward construction of well-prepared
initial data for the phase indicator (see, e.g., \cite{Fischer2020b}), 
the above theorem leads to a sharp $L^1$-convergence rate
for the phase-field and to a sharp $L^2$-convergence rate
for the fluid velocity.

\begin{corollary}[Sharp convergence rates for well-prepared initial data]
\label{cor:sharpConvergenceRates}
Let the assumptions and notation of Theorem~\ref{theo:mainResult} be in place.
For well-prepared initial data in the sense of
\begin{align}
\label{eq:wellPreparedInitialData}
E[\varphi_\eps,v_\eps|\chi,v](0) + 
E_{\mathrm{vol}}[\varphi_\eps|\chi](0)
\leq C(\chi_0)\eps^2,
\end{align}
where $C(\chi_0) \in (0,\infty)$ denotes a constant depending only on
the initial geometry of the strong sharp interface limit solution, it
follows that for a.e.\ $T \in (0,\Tstrong)$ there exists a constant 
$C=C(\chi,v,T) \in (0,\infty)$ such that
\begin{align}
\label{eq:sharpConvergenceRate}
\big\|(v_\eps {-} v)(\cdot,t)\big\|_{L^2(\Omega)}
+ \bigg\|\bigg(c_0\chi(\cdot,t) {-} 
\int_{-1}^{\varphi_\eps(\cdot,t)}\sqrt{2W(r)}\,dr\bigg)\bigg\|_{L^1(\Omega)}
\leq Ce^{Ct}\eps 
\end{align}
for a.e.\ $t \in [0,T]$.
\end{corollary}

The proof of Theorem~\ref{theo:mainResult} is given in Section~\ref{sec:proofMainResult}
and is based on the following key ingredients. As a preliminary, we provide
throughout the remainder of this section precise definitions for
the solution concepts associated with the system of 
PDEs~\eqref{eq:phaseFieldModel1}--\eqref{eq:phaseFieldModel6}
and~\eqref{eq:sharpInterfaceLimit1}--\eqref{eq:sharpInterfaceLimit8}, respectively.
Section~\ref{sec:errorFunctionals} introduces the definition of
the two error functionals for which Theorem~\ref{theo:mainResult}
claims the stability estimates~\eqref{eq:gronwallEstimateRelativeEntropy}
and~\eqref{eq:gronwallEstimatePhaseError}, respectively. 
We also collect several important coercivity properties of these
error functionals. These in turn are needed to appropriately
estimate preliminary bounds for the time evolution of
the two error functionals as derived in Section~\ref{sec:timeEvolRelEntropy}
and Section~\ref{sec:timeEvolWeight}, respectively.

Starting point of our journey is the following precise
concept of weak solutions for the phase-field model.

\begin{definition}[Weak solutions of the Navier--Stokes/Allen--Cahn 
system~\eqref{eq:phaseFieldModel1}--\eqref{eq:phaseFieldModel6}]
\label{def:weakSolPhaseFieldModel}
\begin{subequations}
Let~$d \in \{2,3\}$, let $\Omega\subset\Rd$
be a bounded domain with orientable and $C^2$-boundary,
and consider a finite time horizon $0 < \Tweak < \infty$.
Let~$\eps \in (0,1)$, and consider an initial velocity field $v_{\eps,0} \in H^1(\Omega)$
in combination with an initial order parameter $\varphi_{\eps,0} \in H^1(\Omega)$
such that $\varphi_{\eps,0} \in [-1,1]$ holds true $\mathcal{L}^d$-a.e.\
in~$\Omega$, $\nabla\cdot v_{\eps,0} = 0$ in the distributional sense in~$\Omega$,
as well as $(\varphi_{\eps,0}{-}(-1),v_{\eps,0}) \in H^1_0(\Omega,\mathbb{R}{\times}\Rd)$.

A pair of measurable maps $(\varphi_\eps,v_\eps)\colon\Omega{\times}(0,\Tweak) \to \mathbb{R}{\times}\Rd$
is called a \emph{weak solution of the Navier--Stokes/Allen--Cahn 
system~\emph{\eqref{eq:phaseFieldModel1}--\eqref{eq:phaseFieldModel6}} with time
horizon~$\Tweak$ and initial data~$(\varphi_{\eps,0},v_{\eps,0})$} 
if the following conditions are satisfied:
\begin{itemize}[leftmargin=0.7cm]
\item[i)] In terms of regularity, it has to hold
\begin{align}
\label{eq:conditionsWeakSolPhaseField1}
v_\eps &\in L^2\big(0,\Tweak;H^1(\Omega;\Rd)\big) \cap L^\infty\big(0,\Tweak;L^2(\Omega;\Rd)\big),
\\
\label{eq:conditionsWeakSolPhaseField2}
\varphi_\eps &\in H^1\big(0,\Tweak;L^2(\Omega;[-1,1])\big) \cap L^2\big(0,\Tweak;H^2(\Omega)\big),
\end{align}
such that in addition~\eqref{eq:phaseFieldModel6} holds in form of
\begin{align}
\label{eq:conditionsWeakSolPhaseField3}
\big(\varphi_{\eps}(\cdot,t){-}(-1),v_{\eps}(\cdot,t)\big) \in H^1_0(\Omega,\mathbb{R}{\times}\Rd)
\quad\text{for a.e.\ } t \in (0,\Tweak).
\end{align}
\item[ii)] The velocity field~$v_\eps$ is a solution 
					 of~\eqref{eq:phaseFieldModel1}--\eqref{eq:phaseFieldModel2} 
					 and~\eqref{eq:phaseFieldModel4} in the sense that
					 \begin{align}
					 \label{eq:conditionsWeakSolPhaseField4}
					 \nonumber
					 &\int_{\Omega} v_\eps(\cdot,T) \cdot \eta(\cdot,T) \,dx 
					 - \int_{\Omega} v_{\eps,0} \cdot \eta(\cdot,0) \,dx
					 \\&
					 = \int_0^T\int_{\Omega} v_\eps \cdot \partial_t\eta \,dx dt
					 + \int_0^T\int_{\Omega} v_\eps \otimes v_\eps : \nabla\eta \,dx dt
					 \\&~~~ 
					 \nonumber
					 - \int_0^T\int_{\Omega} \nabla v_\eps : \nabla \eta \,dx dt
					 + \int_0^T\int_{\Omega} \eps\nabla\varphi_\eps\otimes\nabla\varphi_\eps : \nabla \eta \,dx dt
					 \end{align}
					 holds true for a.e.\ $T \in (0,\Tweak)$ and all $\eta \in 
					 C^\infty_c([0,\Tweak);C^\infty_c(\Omega;\Rd))$ with $\nabla\cdot\eta=0$,
					 as well as
					 \begin{align}
					 \label{eq:conditionsWeakSolPhaseField5}
					 (\nabla\cdot v_\eps)(\cdot,t) = 0 
					 \end{align}
					 for a.e.\ $t \in (0,\Tweak)$ in the distributional sense in~$\Omega$.
\item[iii)] The order parameter $\varphi_\eps$ satisfies~\eqref{eq:phaseFieldModel3} 
					  and~\eqref{eq:phaseFieldModel5} in form of
						\begin{align}
						\label{eq:conditionsWeakSolPhaseField6}
						\partial_t\varphi_\eps + (v_\eps\cdot\nabla)\varphi_\eps 
						&= \Delta\varphi_\eps - \smash{\frac{1}{\eps^2}W'(\varphi_\eps)}
						&&\text{a.e.\ in }  \Omega \times (0,\Tweak),
						\\
						\label{eq:conditionsWeakSolPhaseField7}
						\varphi_\eps(\cdot,t) &\to \varphi_{\eps,0} 
						&&\text{strongly in } L^2(\Omega) \text{ as } t \downarrow 0.
						\end{align}
\item[iv)] Defining for every admissible pair~$(\varphi,v)$ the total energy functional by
\begin{align}
\label{eq:energyFunctional}
E[\varphi,v] := \int_{\Omega}\frac{1}{2}|v|^2 \,dx
+ \int_{\Omega} \frac{\eps}{2}|\nabla\varphi|^2 + 
\frac{1}{\eps}W\big(\varphi\big) \,dx, 
\end{align}
sharp energy dissipation is required in form of
\begin{equation}
\label{eq:energyDissipEquality}
\begin{aligned}
&E[\varphi_\eps(\cdot,T),v_\eps(\cdot,T)]
+ \int_0^T\int_{\Omega} \big|\nabla v_\eps\big|^2 
+ \eps\big|\big(\partial_t{+}(v_\eps\cdot\nabla)\big)\varphi_\eps\big|^2 \,dx dt
\\&
\leq E[\varphi_{\eps,0},v_{\eps,0}]
\end{aligned}
\end{equation}
for a.e.\ $T \in (0,\Tweak)$.
\end{itemize}
\end{subequations}
\end{definition}

We conclude this section with a precise concept of strong solutions
for the sharp interface limit model.

\begin{definition}[Strong solutions of the sharp interface limit 
system~\eqref{eq:sharpInterfaceLimit1}--\eqref{eq:sharpInterfaceLimit8}]
\label{def:strongSolSharpInterfaceLimit}
\begin{subequations}
Let~$d \in \{2,3\}$, let $\Omega\subset\Rd$
be a bounded domain with orientable and $C^2$-boundary,
and consider a finite time horizon $0 < \Tstrong < \infty$.
Consider an initial velocity field $v_{0} \in (H^1_0 \cap C^1)(\Omega;\Rd)$
in combination with an initial phase indicator $\chi_{0} \in BV(\Omega;\{0,1\})$
such that $\nabla\cdot v_{0} = 0$ pointwise in~$\Omega$, the set
$\Omega_+(0):=\{\chi_0 {=} 1\}$ yields an open subset of~$\Omega$ consisting of finitely
many connected components with $\overline{\Omega_+(0)} \subset \Omega$,
the associated boundary~$\partial\Omega_+(0)$ is orientable and uniformly of class~$C^{3}$, 
and finally the following compatibility conditions hold true:
\begin{align*}
\jump{ v _0 }=0 ~&\text {on}~ \II(0),\\
(\mathrm{Id} {-} \no_{\II(0)} \otimes \no_{\II(0)})
\jump{    \nabla^\mathrm{sym} v _0  \cdot\no_{\II(0)}  }=0~&\text{on}~\II(0).
\end{align*}

A pair of measurable maps $(\chi,v)\colon\Omega{\times}(0,\Tstrong) \to \{0,1\}{\times}\Rd$
is called a \emph{strong solution of the sharp interface limit 
system~\emph{\eqref{eq:sharpInterfaceLimit1}--\eqref{eq:sharpInterfaceLimit8}} 
with time horizon~$\Tstrong$ and initial data~$(\chi_0,v_0)$} 
if the following requirements hold true:
\begin{itemize}[leftmargin=0.7cm]
\item[i)] The fluid velocity~$v$ is subject to the regularity
					\begin{align}
					\label{eq:conditionsStrongSolSharpInterfaceLimit1}
					v \in H^1\big(0,\Tstrong;L^2(\Omega;\Rd)\big)
					\cap L^2\big(0,\Tstrong;H^1(\Omega;\Rd)\big)
					\end{align}
					such that the corresponding boundary condition~\eqref{eq:sharpInterfaceLimit8}
					as well as the corresponding initial condition~\eqref{eq:sharpInterfaceLimit6}
					are satisfied in form of
					\begin{align}
					\label{eq:conditionsStrongSolSharpInterfaceLimit2}
					v(\cdot,t) &\in H^1_0(\Omega;\Rd) &&\text{for a.e.\ } t \in (0,T_*),
					\\
					\label{eq:conditionsStrongSolSharpInterfaceLimit3}
					v(\cdot,t) &\to v_0 && \text{strongly in } L^2(\Omega) \text{ as } t \downarrow 0.
					\end{align}
					Furthermore, for each $T \in [0,T_*)$ there exists
					$C=C(T) \in (0,\infty)$ such that
					\begin{align}
					\label{eq:conditionsStrongSolSharpInterfaceLimit4}
					v &\in L^\infty\big(0,T;C^{1}(\Omega;\Rd)\big),
					\end{align}
					with a corresponding bound on the highest-order derivative of the form
					\begin{align}
					\label{eq:conditionsStrongSolSharpInterfaceLimit6}
					\|\nabla v\|_{L^\infty(\Omega{\times}(0,T))} \leq C.
					\end{align}
\item[ii)] The velocity field~$v$ solves~\eqref{eq:sharpInterfaceLimit1}--\eqref{eq:sharpInterfaceLimit2}
					 and~\eqref{eq:sharpInterfaceLimit4}--\eqref{eq:sharpInterfaceLimit5} in the sense that
					 \begin{align}
					 \label{eq:conditionsStrongSolSharpInterfaceLimit7}
					 \nonumber
					 &\int_{\Omega} v(\cdot,T) \cdot \eta(\cdot,T) \,dx 
					 - \int_{\Omega} v_0 \cdot \eta(\cdot,0) \,dx
					 \\&
					 = \int_0^T\int_{\Omega} v \cdot \partial_t\eta \,dx dt
					 - \int_0^T\int_{\Omega} \eta \cdot (v \cdot \nabla) v \,dx dt
					 \\&~~~ 
					 \nonumber
					 - \int_0^T\int_{\Omega} \nabla v : \nabla \eta \,dx dt
					 + c_0\int_0^T\int_{\Omega} \frac{\nabla\chi}{|\nabla\chi|} 
					 \otimes \frac{\nabla\chi}{|\nabla\chi|} : \nabla \eta \,d|\nabla\chi| dt
					 \end{align}
					 holds true for a.e.\ $T \in (0,\Tstrong)$ and all $\eta \in 
					 C^\infty_c([0,\Tstrong);C^\infty_c(\Omega;\Rd))$ with $\nabla\cdot\eta=0$,
					 as well as
					 \begin{align}
					 \label{eq:conditionsStrongSolSharpInterfaceLimit8}
					 (\nabla\cdot v)(\cdot,t) = 0 
					 \end{align}
					 for a.e.\ $t \in (0,\Tstrong)$ in the distributional sense 
					 in~$\Omega$. 
\item[iii)] Define for every $t \in (0,\Tstrong)$ the set $\Omega_+(t) := \{\chi(\cdot,t){=}1\}$.
            There then exists a map $\Psi\colon 
						\overline{\Omega} {\times} [0,T_*) \to \overline{\Omega}$ such that
						$\Psi(\cdot,0) = \mathrm{Id}$, $\Psi(\cdot,t) \in 
						C^{3}\text{-Diffeo}(\overline{\Omega},\overline{\Omega})$ for all $t \in (0,\Tstrong)$,
						and $\Omega_+(t) = \Psi(\Omega_+(0),t)$	for all $t \in (0,\Tstrong)$.
						With respect to regularity in time, we require that for all $T \in (0,\Tstrong)$
						it holds
						\begin{align}
						\label{eq:conditionsStrongSolSharpInterfaceLimit9}
						\Psi &\in C^1\big([0,T];C^{1}(\overline{\Omega};\Rd)\big)
						\cap C\big([0,T];C^{3}(\overline{\Omega};\Rd)\big).
						\end{align}
            Note that the corresponding initial condition~\eqref{eq:sharpInterfaceLimit7}
						is thus a consequence of~\eqref{eq:conditionsStrongSolSharpInterfaceLimit9}.
            Finally, the corresponding boundary condition~\eqref{eq:sharpInterfaceLimit8}
						is satisfied in form of
						\begin{align}
						\label{eq:conditionsStrongSolSharpInterfaceLimit10}
						\overline{\Omega_+(t)} &\subset \Omega &&\text{for all } t \in (0,T_*).
						\end{align}
\item[iv)] The geometric evolution equation~\eqref{eq:sharpInterfaceLimit3}
					 is required to be satisfied in its strong form given by~\eqref{eq:geomEvolutionSharpInterface},
					 where the space-time interface $\II={\bigcup_{t \in [0,\Tstrong)} \II(t) {\times} \{t\}}$
					 is defined by $\II(t) := \partial\Omega_+(t)$, $t \in [0,\Tstrong)$.
\end{itemize}
\end{subequations}
\end{definition}

In the appendix, we provide some detailed comments on how to obtain
such strong solutions for the sharp interface limit model starting from the
work of Abels and Moser~\cite{zbMATH06951007}.

\section{The relative energy functional}\label{sec:errorFunctionals}

\subsection{Definition of the relative energy functional}
\label{subsec:relEnergy}
In this subsection, we follow ideas first developed in~\cite{Fischer2020b}.
We start by defining the appropriate $BV$-pendant of~$c_0\chi$ at
the level of the phase-field model via
\begin{align}\label{def psi_eps}
\psi_\eps(x,t) := \int_{-1}^{\varphi_\eps(x,t)} \sqrt{2W}(r) \,dr,
\quad (x,t) \in \Omega \times (0,\Tweak).
\end{align}
We then introduce a ``unit-normal vector field'' $\no_\eps \in \mathbb{S}^{d-1}$ by means of
\begin{align}
\label{eq_defWeakNormal}
\no_\eps(\cdot,t) := 
\begin{cases}
\frac{\nabla\varphi_\eps(\cdot,t)}{|\nabla\varphi_\eps(\cdot,t)|} 
& \text{if } \nabla\varphi_\eps(\cdot,t) \neq 0, \\
\mathrm{s} & \text{else},
\end{cases}
\qquad t \in (0,\Tweak),
\end{align}
with $\mathrm{s}\in\mathbb{S}^{d-1}$ being a fixed but otherwise arbitrary unit vector. 
Note that because of the previous two definitions we always have the relations
\begin{align}
\label{eq:identitiesNormals}
\no_\eps|\nabla\varphi_\eps| = \nabla\varphi_\eps
\quad\text{and}\quad
\no_\eps|\nabla\psi_\eps| = \nabla\psi_\eps
\quad\text{throughout}\quad \Omega \times (0,\Tweak).
\end{align}

We now define a measure for the difference between the phase-field approximation
$(\varphi_\eps,v_\eps)$ and the classical solution $(\chi,v)$ for the sharp interface
limit model as follows: for $t \in (0,\Tstrong) \subset (0,\Tweak)$ let
\begin{align}
\label{eq:relativeEntropyFunctional}
&E[\varphi_\eps,v_\eps|\chi,v](t)
\\&
\nonumber
:= \int_{\Omega} \frac{1}{2}\big|(v_\eps {-} v)(\cdot,t)\big|^2 \,dx 
+ \int_{\Omega} \frac{\eps}{2}\big|\nabla\varphi_\eps(\cdot,t)\big|^2 + 
\frac{1}{\eps}W\big(\varphi_\eps(\cdot,t)\big) 
- (\xi\cdot\nabla\psi_\eps)(\cdot,t) \,dx.
\end{align}
The vector field~$\xi$ associated with the strong sharp interface
limit solution $(\chi,v)$ is yet to be determined.
Note already that the above functional has 
a form reminiscent of that of a relative entropy. Indeed,
\begin{align}
\label{eq:errorRepByRelEntropyStructure}
E[\varphi_\eps,v_\eps|\chi,v] 
= E[\varphi_\eps,v_\eps] - \int_{\Omega} v_\eps\cdot v + \int_{\Omega} \frac{1}{2}|v|^2 \,dx
- \int_{\Omega} \xi\cdot\nabla\psi_\eps \,dx,
\end{align}
where we dropped for notational convenience the dependence on the time variable.
Continuing in this fashion, a straightforward computation moreover shows
\begin{equation}
\label{eq:basicCoercivity}
\begin{aligned}
&E[\varphi_\eps,v_\eps|\chi,v]  
\\
&= \int_{\Omega} \frac{1}{2}|v_\eps - v|^2 +
\int_{\Omega} \frac{1}{2}\Big(\sqrt{\eps}|\nabla\varphi_\eps|
-\frac{1}{\sqrt{\eps}}\sqrt{2W(\varphi_\eps)}\Big)^2 
+\int_{\Omega} (1-\xi\cdot\no_\eps)|\nabla\psi_\eps|.
\end{aligned}
\end{equation}
Ensuring non-negativity of $E[\varphi_\eps,v_\eps|\chi,v]$ thus motivates
$|\xi|\leq 1$ as an absolute minimal assumption on the vector field $\xi$.

The main goal of our approach is to derive a stability estimate on $E[\varphi_\eps,v_\eps\,|\,\chi,v]$
in terms of the initial distance $E_0 := E[\varphi_{\eps,0},v_{\eps,0}|\chi_0,v_0]$ by a Gronwall argument,
cf.\ the stability estimate~\eqref{eq:gronwallEstimateRelativeEntropy} in Theorem~\ref{theo:mainResult}.  
This in turn necessitates control on the time evolution
of the relative energy. In particular, we will need an appropriate control on the time
evolution of the vector field~$\xi$. To this end, it turns out to be beneficial (mostly for clarity of
exposition and efficient organization of terms in the relative energy inequality) 
to introduce a second vector field~$B$. One should keep in mind the following interpretation
of the pair $(\xi,B)$: the vector field~$\xi$ will represent a suitable extension of the 
unit normal vector field~$\no_{\II}$ for the smoothly evolving sharp interface~$\II$,
whereas~$B$ shall denote a suitable extension of the associated
normal velocity vector of~$\II$. More precisely, we impose the following 
conditions on the pair $(\xi,B)$ which will turn out to be sufficient.

For every $T \in [0,\Tstrong)$ there are $c=c(\chi,T) \in (0,1)$
and $C=C(\chi,v,T) \in (0,\infty)$ such that it holds
(employing the convenient notation $\dist(\cdot,\II)$ 
representing the space-time map $\overline{\Omega}{\times}[0,\Tstrong) \ni (x,t)
\mapsto \dist(x,\II(t))$):
\begin{subequations}
\begin{itemize}[leftmargin=0.4cm]
\item (Regularity estimates) In terms of qualitative regularity, it has to hold
\begin{align}
\label{eq:conditionsXiB1}
\xi &\in C^{0,1}\big([0,T];L^\infty(\Omega;\Rd)\big) \cap 
L^\infty\big([0,T];C^{1,1}_c(\Omega;\Rd)\big),
\\
\label{eq:conditionsXiB2}
B &\in L^\infty\big([0,T];C^{0,1}_c(\overline{\Omega};\Rd)\big),
\end{align}
with a bound on the corresponding highest-order  
derivatives in form of
\begin{align}
\label{eq:conditionsXiB3}
\|(\partial_t\xi,\nabla^2\xi,\nabla B)\|_{L^\infty(\Omega{\times}(0,T))} &\leq C.
\end{align}
\item (Coercivity and consistency) It holds 
\begin{align}
\label{eq:conditionsXiB4}
|\xi| &\leq 1{-}c\min\{\dist^2(\cdot,\II),1\}
&&\text{a.e.\ on } \Omega{\times}[0,T]
\\
\label{eq:conditionsXiB5}
\xi &= \no_{\II}
\text{ and }
\nabla\cdot \xi = -H_{\II} 
&&\text{on } \II.
\end{align}
\item (Approximate transport of $\xi$ by $B$) We have
\begin{align}
\label{eq:conditionsXiB6}
\big|\partial_t\xi + (B\cdot\nabla)\xi + (\nabla B)^\mathsf{T}\xi\big| 
&\leq C\min\{\dist(\cdot,\II),1\}
&&\text{a.e.\ on } \Omega{\times}[0,T],
\\
\label{eq:conditionsXiB7}
\big|\xi\cdot(\partial_t + B\cdot\nabla)\xi\big| 
&\leq C\min\{\dist^2(\cdot,\II),1\}
&&\text{a.e.\ on } \Omega{\times}[0,T].
\end{align} 
\item (Interpretation of $B$ as a normal velocity) It holds
\begin{align}
\label{eq:conditionsXiB8}
\big|(B{-}v)\cdot\xi + \nabla\cdot\xi\big| 
&\leq C\min\{\dist(\cdot,\II),1\}
&&\text{a.e.\ on } \Omega{\times}[0,T],
\\
\label{eq:conditionsXiB9}
\big|\xi\cdot(\xi\cdot\nabla) B \big|
&\leq C\min\{\dist(\cdot,\II),1\}
&&\text{a.e.\ on } \Omega{\times}[0,T].
\end{align}
\end{itemize}
\end{subequations}

\subsection{Coercivity properties of the relative energy functional}
We collect several useful coercivity estimates for $E[\varphi_\eps,v_\eps|\chi,v]$,
dropping in the process again for notational convenience the dependence
on the time variable. The proof of these estimates is provided 
afterwards. 
 
\begin{lemma}
Let the assumptions and notation of Subsection~\ref{subsec:relEnergy} be in place.
First, the relative energy~$E[\varphi_\eps,v_\eps|\chi,v]$ 
provides control on the error in the fluid velocity by means of
			\begin{align}
			\label{eq:coercivityRelEntropy1}
			\int_{\Omega} \frac{1}{2} |v_\eps-v|^2 \,dx \leq E[\varphi_\eps,v_\eps|\chi,v].
			\end{align}
Second, it entails a ``tilt-excess'' type control on the geometry in the form of
			\begin{align}
			\label{eq:coercivityRelEntropy2}
			\int_{\Omega} (1-\no_\eps\cdot\xi)|\nabla\psi_\eps| \,dx \leq E[\varphi_\eps,v_\eps|\chi,v].
			\end{align}
			Third, one obtains control on the lack of equipartition of energy through
						\begin{align}
			\label{eq:coercivityRelEntropy3}
			\int_{\Omega} \frac{1}{2}\Big(\sqrt{\eps}|\nabla\varphi_\eps|
			-\frac{1}{\sqrt{\eps}}\sqrt{2W(\varphi_\eps)}\Big)^2 \,dx
			\leq E[\varphi_\eps,v_\eps|\chi,v].
			\end{align}
Fourth, for every $T \in [0,\Tstrong)$ there exists
a constant $C = C(\chi,v,T) \in (0,\infty)$ such that
for all times in $[0,T]$ it holds
			\begin{align}
			\label{eq:coercivityRelEntropy4}
			\int_{\Omega} |\no_\eps-\xi|^2|\nabla\psi_\eps| \,dx
			+ \int_{\Omega} \min\{\dist^2(\cdot,\II),1\} |\nabla\psi_\eps| \,dx
			&\leq CE[\varphi_\eps,v_\eps|\chi,v],
			\\
			\label{eq:coercivityRelEntropy5}
			\int_{\Omega} |\no_\eps-\xi|^2\eps|\nabla\varphi_\eps|^2 \,dx
			+ \int_{\Omega} \min\{\dist^2(\cdot,\II),1\} \eps|\nabla\varphi_\eps|^2 \,dx
			&\leq CE[\varphi_\eps,v_\eps|\chi,v],
			\end{align}
as well as finally
			\begin{align}
			\label{eq:coercivityRelEntropy6}
			\int_{\Omega} \big(\min\{\dist(\cdot,\II),1\}+\sqrt{1 {-} \no_\eps\cdot\xi} \big)
			\big|\eps|\nabla\varphi_\eps|^2-|\nabla\psi_\eps|\big| \,dx
			\leq CE[\varphi_\eps,v_\eps|\chi,v].
			\end{align}
\end{lemma}

\begin{proof}
The assertions~\eqref{eq:coercivityRelEntropy2}--\eqref{eq:coercivityRelEntropy6} 
are essentially already contained
and proved in~\cite[Subsection~2.3]{Fischer2020b}. For the sake of completeness,
we re-produce the straightforward argument here.

By~\eqref{def psi_eps} and~\eqref{eq_defWeakNormal}, we can 
write~\eqref{eq:relativeEntropyFunctional} as the sum of three integrals 
with each featuring a non-negative integrand:  
\begin{align*}
E[\varphi_\eps,v_\eps|\chi,v](t)
&= \int_{\Omega} \frac{1}{2}\big|(v_\eps {-} v)(\cdot,t)\big|^2 \,dx  \\
&+\int_{\Omega}|\nabla \psi_\eps(\cdot,t)|
- (\xi\cdot\nabla\psi_\eps)(\cdot,t) \,dx \\
&+ \int_{\Omega} \frac{\eps}{2}\big|\nabla\varphi_\eps(\cdot,t)\big|^2 + 
\frac{1}{\eps}W\big(\varphi_\eps(\cdot,t)\big) -|\nabla \psi_\eps(\cdot,t)| \, dx.
\end{align*}
Using~\eqref{def psi_eps} and the chain rule, we can complete the square in 
the last integral above, and this proves~\eqref{eq:coercivityRelEntropy1}, 
\eqref{eq:coercivityRelEntropy2} and~\eqref{eq:coercivityRelEntropy3}.
The estimate~\eqref{eq:coercivityRelEntropy4} follows from ~\eqref{eq:conditionsXiB4} 
and~\eqref{eq:coercivityRelEntropy2}.

Next, adding zero and employing Young's inequality in the form of
$$
\begin{aligned}
\	\eps\left|\nabla \varphi_\eps\right|^{2} &=\left|\nabla \psi_\eps\right|+\sqrt{\varepsilon}\left|\nabla \varphi_\eps\right|\left(\sqrt{\varepsilon}\left|\nabla \varphi_\eps\right|-\frac{1}{\sqrt{\varepsilon}} \sqrt{2 W\left(\varphi_\eps\right)}\right) \\
& \leq\left|\nabla \psi_\eps\right|+\frac{1}{2} \varepsilon\left|\nabla \varphi_\eps\right|^{2}+\frac{1}{2}\left(\sqrt{\varepsilon}\left|\nabla \varphi_\eps\right|-\frac{1}{\sqrt{\varepsilon}} \sqrt{2 W\left(\varphi_\eps\right)}\right)^{2}
\end{aligned}
$$
we infer 
\begin{align}\label{minusorcali}
&\int_\O(1-\no_\eps\cdot\xi) \varepsilon\left|\nabla \varphi_\eps\right|^{2} d x\nonumber \\
&~~~\leq  2 \int_\O(1-\no_\eps\cdot\xi)\left|\nabla \psi_\eps\right| \,d x+4 \int_\O\left(\sqrt{\varepsilon}\left|\nabla \varphi_\eps\right|-\frac{1}{\sqrt{\varepsilon}} \sqrt{2 W\left(\varphi_\eps\right)}\right)^{2} \,dx\nonumber\\
&\overset{\eqref{eq:coercivityRelEntropy2},\eqref{eq:coercivityRelEntropy3}}\leq   
CE[\varphi_\eps,v_\eps|\chi,v],
\end{align}
which then together with $2(1-\no_\eps\cdot\xi )\geq |\no_\eps-\xi|^2$ 
and~\eqref{eq:conditionsXiB4} leads to~\eqref{eq:coercivityRelEntropy5}. Finally 
\begin{align*}
		&	\int_{\Omega} \min\{\dist(\cdot,\II),1\}
			\big|\eps|\nabla\varphi_\eps|^2-|\nabla\psi_\eps|\big| \,dx\\
			&\overset{\eqref{eq:conditionsXiB4}}\leq  C\int_{\Omega}  \sqrt{1-\no_\eps\cdot\xi}\sqrt{\varepsilon}\left|\nabla \varphi_\eps\right|\left|\sqrt{\varepsilon}\left|\nabla \varphi_\eps\right|-\frac{1}{\sqrt{\varepsilon}} \sqrt{2 W\left(\varphi_\eps\right)}\right|\, dx.
			\end{align*}
So \eqref{eq:coercivityRelEntropy6} is a consequence of \eqref{minusorcali} and the Cauchy--Schwarz inequality.
\end{proof}

\subsection{Error in the phase indicators}\label{subsec:volError}
The second step in the proof consists of deriving an estimate
for the error in the phase indicators~$\varphi_\eps$ and~$\chi$ 
based on the stability estimate for the relative energy $E[\varphi_\eps,v_\eps|\chi,v]$.
This in turn is done by means of a stability estimate \`a la 
Gronwall (cf.\ the estimate \eqref{eq:gronwallEstimatePhaseError} in Theorem~\ref{theo:mainResult}) 
for an appropriate error functional defined for all $t \in [0,T_*)$ as follows:
\begin{align}
\label{eq:defWeightVol}
E_{\mathrm{vol}}[\varphi_\eps|\chi](t) := 
\int_{\Omega} \big|\psi_\eps(\cdot,t) - c_0\chi(\cdot,t)\big| \big|\vartheta(\cdot,t)\big| \,dx.
\end{align}
Here, $\vartheta\colon\overline{\Omega}{\times}[0,T_*)\to [-1,1]$ denotes
a (by the velocity field~$B$ transported) time-dependent weight function
satisfying the following conditions, which in turn will be sufficient to derive
a stability estimate for~$E_{\mathrm{vol}}[\varphi_\eps|\chi]$.

For every $T \in [0,\Tstrong)$ there are $c=c(\chi,T) \in (0,1)$
and $C=C(\chi,v,T) \in (1,\infty)$ such that it holds
(employing again the convenient notation $\dist(\cdot,\II)$ 
representing the space-time map $\overline{\Omega}{\times}[0,\Tstrong) \ni (x,t)
\mapsto \dist(x,\II(t))$):
\begin{subequations}
\begin{itemize}[leftmargin=0.5cm]
\item (Regularity estimates) We have
			\begin{align}
			\label{eq:conditionsWeight1}
			\vartheta \in C^{0,1}\big([0,T];L^\infty(\overline{\Omega})\big) 
										\cap L^\infty\big([0,T];C^{0,1}(\overline{\Omega})\big)
			\end{align}
			with a corresponding estimate for the highest-order derivatives in form of
			\begin{align}
			\label{eq:conditionsWeight2}
			\|(\partial_t\vartheta,\nabla\vartheta)\|_{L^\infty(\Omega{\times}(0,T))} &\leq C.
			\end{align}
\item (Coercivity and consistency) It holds
			\begin{align}
			\label{eq:conditionsWeight3}
			c\min\{\dist(\cdot,\II),1\} \leq |\vartheta| &\leq C\min\{\dist(\cdot,\II),1\}
			&&\text{a.e.\ on } \Omega{\times}[0,T],
			\end{align}
			as well as for all $t \in [0,\Tstrong)$
			\begin{align}
			\label{eq:conditionsWeight4}
			\vartheta(\cdot,t) &< 0 
			&&\text{a.e.\ in the interior of } \{\chi(\cdot,t) {=} 1\} \cap \Omega,
			\\
			\label{eq:conditionsWeight5}
			\vartheta(\cdot,t) &> 0 
			&&\text{a.e.\ in the interior of } \{\chi(\cdot,t) {=} 0\} \cap \Omega.
			\end{align}
\item (Transport equation) Finally, it is required that
		  \begin{align}
			\label{eq:conditionsWeight6}
			\big|\partial_t\vartheta + (B\cdot\nabla)\vartheta\big|
			&\leq C\min\{\dist(\cdot,\II),1\}
			&&\text{a.e.\ on } \Omega{\times}[0,T].
			\end{align}
\end{itemize}
\end{subequations}

As in the case of the relative energy functional,
we conclude the discussion of this subsection 
by recording useful coercivity properties of
the error functional~$E_{\mathrm{vol}}[\varphi_\eps|\chi]$
(dropping yet again the dependence on the time variable).
More precisely, for every $T \in [0,\Tstrong)$ there exists
a constant $C = C(\chi,v,T) \in (0,\infty)$ such that
for all times in $[0,T]$ it obviously holds that
\begin{align}
\label{eq:coercivityBulkError1}
\int_{\Omega} \min\{\dist(\cdot,\II),1\}|\psi_\eps - c_0\chi| \,dx 
\leq CE_{\mathrm{vol}}[\varphi_\eps|\chi].
\end{align}
More importantly, it holds for all $\lambda\in (0,1)$
\begin{align}
\label{eq:coercivityBulkError2}
\int_{\Omega} |c_0\chi {-} \psi_\eps||v_\eps {-} v| \,dx
&\leq \frac{C}{\lambda} \big(E_{\mathrm{vol}}[\varphi_\eps|\chi] {+} E[\varphi_\eps,v_\eps|\chi,v]\big)
+ \lambda\int_{\Omega} |\nabla v_\eps {-} \nabla v|^2 \,dx.
\end{align}

\begin{proof}[Proof of the coercivity estimate~\eqref{eq:coercivityBulkError2}]
Let $\delta=\delta(\chi,T) \in (0,\smash{\frac{1}{2}}]$ be the thickness 
of the tubular neighborhood $B_{\delta}(\II(t))
:= \{x \in \Rd\colon \dist(x,\II(t)) < \delta\}$ of $\II(t)$. For fixed $t\in [0,T]$, we have 
\begin{align*}
&\int_{\O \setminus B_{\delta}(\II(t))} |c_0\chi {-} \psi_\eps||v_\eps {-} v| \,dx\\
&~~\overset{\eqref{eq:coercivityBulkError1}}\leq  \frac 1 \delta\int_\O \min\{\dist(\cdot,\II),1\} |c_0\chi {-} \psi_\eps||v_\eps {-} v| \,dx\\
&~~\overset{\eqref{eq:conditionsWeakSolPhaseField2}}\leq  \frac C \delta\int_\O \sqrt{\min\{\dist(\cdot,\II),1\}} |c_0\chi {-} \psi_\eps|^{1/2}|v_\eps {-} v| \,dx\\
&\overset{\eqref{eq:coercivityBulkError1},\,\eqref{eq:coercivityRelEntropy1}}\leq  \frac{C}{\delta} \big(E_{\mathrm{vol}}[\varphi_\eps|\chi] {+} E[\varphi_\eps,v_\eps|\chi,v]\big).
\end{align*}
It remains to estimate the integral over $B_{\delta}(\II(t))$, and  we shall use the following elementary inequality
\begin{align}\label{def: sobolev 1d}
\left(\int_{0}^{\tau} g(r) d r\right)^{2} \leq 2\|g\|_{L^{\infty}(0, \tau)} \int_{0}^{\tau} r g(r)   d r, \quad \forall g(r) \geq 0, \tau \geq 0.
\end{align}
Recall that $\mathbf{n}_{\II}(y,t)$ is the inward normal vector at $y\in \II(t)$.
By the area formula and the one-dimensional Gagliardo--Nirenberg--Sobolev inequality, 
\begin{align*}
&\int_{B_{\delta}(\II(t))} |c_0\chi {-} \psi_\eps||v_\eps {-} v| \,dx\\
&~\leq C \int_{\II(t)} \int_{-\delta}^\delta |c_0\chi {-} \psi_\eps||v_\eps {-} v|
(y{+}\mathbf{n}_{\II}(y,t) r)\, d r d \mathcal{H}^{d-1}(y)\\
&~\leq C \int_{\II(t)} \sup _{|r| \leq  \delta}|v_\eps{-}v| (y{+}\mathbf{n}_{\II}(y,t) r)\left(\int_{-\delta}^\delta\left|c_0\chi {-} \psi_\eps\right|(y{+}\no_{\II}(y,t) r) d r\right) d \mathcal{H}^{d-1}(y)\\
&\stackrel{\eqref{def: sobolev 1d}}{\leq } C \int_{\II(t)}\left\|(v_\eps{-}v) (y{+}\no_{\II} r)\right\|_{L^{2}((- \delta,  \delta);dr)}^{\frac{1}{2}}\left\| (v_\eps{-}v) (y{+}\no_{\II} r)\right\|_{H^{1}((- \delta,  \delta);dr)}^{\frac{1}{2}} \\
&\qquad~~~~~~~~ \cdot \sqrt{\int_{-\delta}^\delta\left|r \| c_0\chi{-}\psi_\eps\right|(y{+}\no_{\II} r) \,d r } \,d \mathcal{H}^{d-1}(y).
\end{align*}
By H\"{o}lder's and Young's inequalities, the terms in the last step can be controlled 
by the right hand side of~\eqref{eq:coercivityBulkError2}.
\end{proof}

\section{Estimate on the time evolution of the relative energy}\label{sec:timeEvolRelEntropy}
The aim of this section is to derive an estimate representing the key ingredient
for the derivation of the asserted stability estimate~\eqref{eq:gronwallEstimateRelativeEntropy} 
for the relative energy.

\begin{proposition}
\label{prop:relEntropyInequality}
Let the assumptions and notation of Theorem~\ref{theo:mainResult}
be in place, let~$(\xi,B)$ be a pair of vector fields
subject to~\emph{\eqref{eq:conditionsXiB1}--\eqref{eq:conditionsXiB3}} 
as well as~\eqref{eq:conditionsXiB5}, and let the relative energy $E[\varphi_\eps,v_\eps|\chi,v]$
be given by~\eqref{eq:relativeEntropyFunctional}. Define the quantity
\begin{align}
\label{eq:phaseFieldCurvature}
H_\eps := -\eps\Delta\varphi_\eps + \frac{1}{\eps}W'(\varphi_\eps)
\quad \text{in } \Omega \times (0,\Tweak).
\end{align}
(This notation serves as a reminder that $H_\eps$ plays the role of a scalar curvature.)
Then, the following relative energy inequality holds true
\begin{align*}
&E[\varphi_\eps,v_\eps|\chi,v](T)
\\
&\leq E[\varphi_\eps,v_\eps|\chi,v](0)
- \int_{0}^{T}\int_{\Omega} |\nabla v_\eps - \nabla v|^2 \,dx dt
\\&~~~
- \int_{0}^{T}\int_{\Omega} \frac{1}{2\eps} 
	\Big| H_\eps + \sqrt{2W(\varphi_\eps)}(\nabla\cdot\xi) \Big|^2 \,dx dt
\\&~~~
- \int_{0}^{T}\int_{\Omega} \frac{1}{2\eps} 
  \Big| H_\eps - \big((B - v)\cdot\xi\big)\eps|\nabla\varphi_\eps| \Big|^2 \,dx dt
\\&~~~
- \int_{0}^{T}\int_{\Omega} (v_\eps - v)\cdot\big((v_\eps-v)\cdot\nabla\big)v \,dx dt
\\&~~~
- \int_{0}^{T}\int_{\Omega} (c_0\chi - \psi_\eps) \big((v_\eps-v)\cdot\nabla\big)(\nabla\cdot\xi) \,dx dt
\\&~~~
+ \int_{0}^{T}\int_{\Omega} \big|(B-v)\cdot\xi + \nabla\cdot\xi\big|^2\eps|\nabla\varphi_\eps|^2 \,dx dt
\\&~~~
+ \int_{0}^{T}\int_{\Omega} |\nabla\cdot\xi|^2
\bigg(\frac{\sqrt{2W(\varphi_\eps)}}{\sqrt{\eps}}-\sqrt{\eps}|\nabla\varphi_\eps|\bigg)^2 \,dx dt
\\&~~~
- \int_{0}^{T}\int_{\Omega} \frac{1}{\sqrt{\eps}}\big(H_\eps + \sqrt{2W(\varphi_\eps)}(\nabla\cdot\xi)\big)
\big((v-B)\cdot(\no_\eps-\xi)\big)\sqrt{\eps}|\nabla\varphi_\eps| \,dx dt
\\&~~~
-\int_{0}^{T}\int_{\Omega} \big(\partial_t\xi + (B\cdot\nabla)\xi + (\nabla B)^\mathsf{T}\xi\big)
\cdot(\no_\eps-\xi)|\nabla\psi_\eps| \,dx dt
\\&~~~
-\int_{0}^{T}\int_{\Omega} \xi\cdot\big(\partial_t\xi+(B\cdot\nabla)\xi\big) |\nabla\psi_\eps| \,dx dt
\\&~~~
-\int_{0}^{T}\int_{\Omega} \nabla B:(\xi-\no_\eps)\otimes(\xi-\no_\eps) |\nabla\psi_\eps| \,dx dt
\\&~~~
+ \int_{0}^{T}\int_{\Omega} (\nabla\cdot B)(1-\xi\cdot\no_\eps)|\nabla\psi_\eps| \,dx dt
\\&~~~
+ \int_{0}^{T}\int_{\Omega} (\nabla\cdot B) \frac{1}{2}\Big(\sqrt{\eps}|\nabla\varphi_\eps|
-\frac{1}{\sqrt{\eps}}\sqrt{2W(\varphi_\eps)}\Big)^2 \,dx dt
\\&~~~
-\int_{0}^{T}\int_{\Omega} \big(\no_\eps\otimes\no_\eps - \xi\otimes\xi\big) : \nabla B\, 
\big(\eps|\nabla\varphi_\eps|^2 - |\nabla\psi_\eps|\big) \,dx dt
\\&~~~
-\int_{0}^{T}\int_{\Omega} \xi\otimes\xi : \nabla B\, 
\big(\eps|\nabla\varphi_\eps|^2 - |\nabla\psi_\eps|\big) \,dx dt
\end{align*}
for a.e.\ $T \in (0,\Tstrong)$.
\end{proposition}

\begin{proof}
We split the proof into six steps.

\textit{Step 1: Time evolution of kinetic energy terms.}
We claim that
\begin{align}
\nonumber
&\int_{\Omega} \frac{1}{2} \big|(v_\eps {-} v)(\cdot,T)\big|^2 \,dx
- \int_{\Omega} \frac{1}{2} \big|v_{\eps,0} {-} v_0\big|^2 \,dx 
\\& \label{eq:auxRelEntropyTimeEvol1}
= \int_{\Omega} \frac{1}{2} |v_\eps(\cdot,T)|^2 \,dx
- \int_{\Omega} \frac{1}{2} |v_{\eps,0}|^2 \,dx
+ \int_0^T\int_{\Omega} |\nabla v_\eps|^2 \,dx dt
\\&~~~ \nonumber
- \int_0^T\int_{\Omega} \big|\nabla (v_\eps {-} v)\big|^2 \,dx dt
- \int_0^T\int_{\Omega} (v_\eps {-} v) \cdot \big((v_\eps {-} v)\cdot\nabla\big)v \,dx dt
\\&~~~ \nonumber
- \int_0^T\int_{\Omega} \eps\nabla\varphi_\eps \cdot (\nabla\varphi_\eps \cdot \nabla) v \,dx dt
- \int_0^T\int_{\Omega} c_0\chi ((v_\eps {-} v)\cdot\nabla)(\nabla\cdot\xi) \,dx dt
\end{align}
for a.e.\ $T \in (0,\Tstrong)$. For a proof of~\eqref{eq:auxRelEntropyTimeEvol1},
we first observe that by~\eqref{eq:conditionsStrongSolSharpInterfaceLimit7}
and the regularity of the sharp interface limit solution it holds
\begin{equation}
\label{eq:auxRelEntropyTimeEvol2}
\begin{aligned}
0 &= \int_0^T\int_{\Omega} \eta \cdot \big(\partial_t + (v\cdot\nabla)\big)v \,dx dt
+ \int_0^T\int_{\Omega} \nabla v : \nabla \eta \,dx dt
\\&~~~
- c_0 \int_0^T\int_{\II} H_{\II}\no_{\II} \cdot \eta \,d\mathcal{H}^{d-1} dt
\end{aligned}
\end{equation}
for a.e.\ $T \in (0,\Tstrong)$ and all $\eta \in C^\infty_c(\Omega{\times}[0,T];\Rd)$ 
with $\nabla\cdot\eta=0$. Instead of directly exploiting the information 
provided by~\eqref{eq:auxRelEntropyTimeEvol2}, it is convenient to post-process the curvature term before.
More precisely, for any admissible solenoidal test vector field in~\eqref{eq:auxRelEntropyTimeEvol2},
it follows from plugging in the second identity of~\eqref{eq:conditionsXiB5} as well as
integrating by parts that
\begin{align*}
- c_0 \int_0^T\int_{\II} H_{\II}\no_{\II} \cdot \eta \,d\mathcal{H}^{d-1} dt
&= c_0 \int_0^T\int_{\II} (\nabla\cdot\xi)\no_{\II} \cdot \eta \,d\mathcal{H}^{d-1} dt
\\&
= - \int_0^T\int_{\Omega} c_0\chi (\eta\cdot\nabla)(\nabla\cdot\xi) \,dx dt.
\end{align*}
Hence, \eqref{eq:auxRelEntropyTimeEvol2} updates to
\begin{equation}
\label{eq:auxRelEntropyTimeEvol3}
\begin{aligned}
0 &= \int_0^T\int_{\Omega} \eta \cdot \big(\partial_t + (v\cdot\nabla)\big)v \,dx dt
+ \int_0^T\int_{\Omega} \nabla v : \nabla \eta \,dx dt
\\&~~~
- \int_0^T\int_{\Omega} c_0\chi (\eta\cdot\nabla)(\nabla\cdot\xi) \,dx dt
\end{aligned}
\end{equation}
for a.e.\ $T \in (0,\Tstrong)$ and all $\eta \in C^\infty_c(\Omega{\times}[0,T];\Rd)$ 
with $\nabla\cdot\eta=0$. By means of standard mollification arguments, one
may plug in the choice $\eta = v_\eps - v$ into~\eqref{eq:auxRelEntropyTimeEvol3}
resulting in the identity
\begin{equation}
\label{eq:auxRelEntropyTimeEvol4}
\begin{aligned}
0 &= \int_0^T\int_{\Omega} (v_\eps {-} v) \cdot \big(\partial_t + (v\cdot\nabla)\big)v \,dx dt
+ \int_0^T\int_{\Omega} \nabla v : \nabla (v_\eps {-} v) \,dx dt
\\&~~~
- \int_0^T\int_{\Omega} c_0\chi ((v_\eps {-} v)\cdot\nabla)(\nabla\cdot\xi) \,dx dt
\end{aligned}
\end{equation}
for a.e.\ $T \in (0,\Tstrong)$. Next to the previous display, we also trivially have
by the regularity of the sharp interface limit solution as well as the solenoidality
of the velocity fields~$v_\eps$ and~$v$, respectively, that
\begin{equation}
\label{eq:auxRelEntropyTimeEvol5}
\begin{aligned}
&\int_{\Omega} \frac{1}{2}|v(\cdot,T)|^2 \,dx - \int_{\Omega} \frac{1}{2}|v_0|^2 \,dx
\\&
= \int_0^T\int_{\Omega} v \cdot \partial_t v \,dx dt
= \int_0^T\int_{\Omega} v \cdot \partial_t v \,dx dt
+ \int_0^T\int_{\Omega} v \cdot (v_\eps \cdot \nabla) v \,dx dt
\end{aligned}
\end{equation}
for a.e.\ $T \in (0,\Tstrong)$. Finally, up to a standard mollification 
argument one may insert $\eta=v$ as a test vector field in~\eqref{eq:conditionsWeakSolPhaseField4}
entailing
\begin{align}
\nonumber
&\int_{\Omega} v_\eps(\cdot,T) \cdot v(\cdot,T) \,dx 
- \int_{\Omega} v_{\eps,0} \cdot v_0 \,dx
\\& \label{eq:auxRelEntropyTimeEvol6}
= \int_0^T\int_{\Omega} v_\eps \cdot \partial_t v \,dx dt
+ \int_0^T\int_{\Omega} v_\eps \cdot (v_\eps \cdot \nabla) v \,dx dt
\\&~~~ 
\nonumber
- \int_0^T\int_{\Omega} \nabla v_\eps : \nabla v \,dx dt
+ \int_0^T\int_{\Omega} \eps\nabla\varphi_\eps\otimes\nabla\varphi_\eps : \nabla v \,dx dt
\end{align}
for a.e.\ $T \in (0,\Tstrong)$. Hence, adding first~\eqref{eq:auxRelEntropyTimeEvol4}
to~\eqref{eq:auxRelEntropyTimeEvol5} and then subtracting~\eqref{eq:auxRelEntropyTimeEvol6}
from the resulting identity produces the claim~\eqref{eq:auxRelEntropyTimeEvol1}.

\textit{Step 2: Time evolution of interfacial energy terms.}
We next claim that
\begin{align}
\nonumber
&\int_{\Omega} \frac{\eps}{2}\big|\nabla\varphi_\eps(\cdot,T)\big|^2 + 
\frac{1}{\eps}W\big(\varphi_\eps(\cdot,T)\big) 
- (\xi\cdot\nabla\psi_\eps)(\cdot,T) \,dx
\\&~~~ \nonumber
- \int_{\Omega} \frac{\eps}{2}\big|\nabla\varphi_{\eps,0}\big|^2 + 
\frac{1}{\eps}W\big(\varphi_{\eps,0}\big) 
- (\xi(\cdot,0)\cdot\nabla\psi_{\eps,0}) \,dx
\\& \label{eq:auxRelEntropyTimeEvol7}
= \int_{\Omega} \frac{\eps}{2}\big|\nabla\varphi_\eps(\cdot,T)\big|^2 {+} 
\frac{1}{\eps}W\big(\varphi_\eps(\cdot,T)\big) \,dx
- \int_{\Omega} \frac{\eps}{2}\big|\nabla\varphi_{\eps,0}\big|^2 {+} 
\frac{1}{\eps}W\big(\varphi_{\eps,0}\big) \,dx
\\&~~~ \nonumber
+ \int_0^T\int_{\Omega} \frac{1}{\eps}|H_\eps|^2 \,dxdt
\\&~~~ \nonumber
- \int_0^T\int_{\Omega} \frac{1}{\eps}|H_\eps|^2 \,dxdt
- \int_0^T\int_{\Omega} (\nabla\cdot\xi)\frac{H_\eps}{\sqrt{\eps}}
	\frac{\sqrt{2W(\varphi_\eps)}}{\sqrt{\eps}}  \,dx dt
\\&~~~ \nonumber
- \int_0^T\int_{\Omega} (\nabla\cdot\xi)(v_\eps\cdot\nabla)\psi_\eps \,dx dt
- \int_0^T\int_{\Omega} \partial_t\xi \cdot \nabla\psi_\eps \,dx dt
\end{align}
for a.e.\ $T \in (0,\Tstrong)$. For a proof of~\eqref{eq:auxRelEntropyTimeEvol7},
we start with an integration by parts which combined with the fundamental theorem
of calculus, the regularity of the phase-field~$\varphi_\eps$,
a standard mollifier approximation (w.r.t.\ the spatial variable)~$\xi_k:=
\theta_k\ast\xi$ for the vector field~$\xi$ (recall for this the regularity 
requirements~\eqref{eq:conditionsXiB1} and~\eqref{eq:conditionsXiB3}), 
the condition~\eqref{eq:conditionsWeakSolPhaseField7}, and finally
another integration by parts ensures
\begin{equation}
\label{eq:auxRelEntropyTimeEvol8}
\begin{aligned}
&\int_{\Omega} (\xi\cdot\nabla\psi_\eps)(\cdot,T) \,dx
- \int_{\Omega} \xi(\cdot,0)\cdot\nabla\psi_{\eps,0} \,dx
\\&
= - \bigg(\int_{\Omega} (\nabla\cdot\xi)(\cdot,T)\psi_\eps(\cdot,T) \,dx
- \int_{\Omega} (\nabla\cdot\xi)(\cdot,0)\psi_{\eps,0} \,dx \bigg)
\\&
= - \lim_{k\to\infty}\int_0^T\int_{\Omega} (\nabla\cdot\partial_t\xi_k)\psi_\eps \,dx dt
- \int_0^T\int_{\Omega} (\nabla\cdot\xi) \partial_t\psi_\eps \,dx dt
\\&
= \int_0^T\int_{\Omega} \partial_t\xi\cdot\nabla\psi_\eps \,dx dt
- \int_0^T\int_{\Omega} (\nabla\cdot\xi) \partial_t\psi_\eps \,dx dt.
\end{aligned}
\end{equation}
By the chain rule, the condition~\eqref{eq:conditionsWeakSolPhaseField6},
and the definition~\eqref{eq:phaseFieldCurvature}, it also follows
\begin{align}
\label{eq:evolPsi}
\partial_t\psi_\eps = 
-\frac{H_\eps}{\sqrt{\eps}}\frac{\sqrt{2W(\varphi_\eps)}}{\sqrt{\eps}} 
- (v_\eps\cdot\nabla)\psi_\eps.
\end{align}
Hence, one obtains the claim~\eqref{eq:auxRelEntropyTimeEvol7}
from plugging in~\eqref{eq:evolPsi} into~\eqref{eq:auxRelEntropyTimeEvol8}
and adding (in a self-evident way) several zeros.

\textit{Step 3: Exploiting sharp energy dissipation.}
The combination of the two identities~\eqref{eq:auxRelEntropyTimeEvol1}
and~\eqref{eq:auxRelEntropyTimeEvol7} from the previous two steps
together with the sharp energy dissipation estimate~\eqref{eq:energyDissipEquality},
the condition~\eqref{eq:conditionsWeakSolPhaseField6},
and finally the definition~\eqref{eq:phaseFieldCurvature} entail
\begin{align}
\nonumber
&E[\varphi_\eps,v_\eps|\chi,v](T)
\\ \nonumber
&\leq E[\varphi_\eps,v_\eps|\chi,v](0)
- \int_{0}^{T}\int_{\Omega} |\nabla v_\eps - \nabla v|^2 \,dx dt
\\&~~~ \nonumber
- \int_0^T\int_{\Omega} (v_\eps {-} v) \cdot \big((v_\eps {-} v)\cdot\nabla\big)v \,dx dt
\\&~~~ \nonumber
- \int_0^T\int_{\Omega} \frac{1}{\eps}|H_\eps|^2 \,dxdt
- \int_0^T\int_{\Omega} (\nabla\cdot\xi)\frac{H_\eps}{\sqrt{\eps}}
	\frac{\sqrt{2W(\varphi_\eps)}}{\sqrt{\eps}}  \,dx dt
\\&~~~ \nonumber
- \int_0^T\int_{\Omega} c_0\chi ((v_\eps {-} v)\cdot\nabla)(\nabla\cdot\xi) \,dx dt
- \int_0^T\int_{\Omega} (\nabla\cdot\xi)(v_\eps\cdot\nabla)\psi_\eps \,dx dt
\\&~~~ \nonumber
- \int_0^T\int_{\Omega} \eps\nabla\varphi_\eps \cdot (\nabla\varphi_\eps \cdot \nabla) v \,dx dt
- \int_0^T\int_{\Omega} \partial_t\xi\cdot\nabla\psi_\eps \,dx dt
\\& \label{eq:auxRelEntropyTimeEvol9}
=:  E[\varphi_\eps,v_\eps|\chi,v](0)
- \int_{0}^{T}\int_{\Omega} |\nabla v_\eps - \nabla v|^2 \,dx dt
\\&~~~~ \nonumber
- \int_0^T\int_{\Omega} (v_\eps {-} v) \cdot \big((v_\eps {-} v)\cdot\nabla\big)v \,dx dt
+ \mathrm{Res}^{(1)}(T).
\end{align}
The remainder of the proof takes care of suitably rewriting the residual term~$\mathrm{Res}^{(1)}$.
To this end, we start with an auxiliary result.

\textit{Step 4: A well-known identity for the phase-field curvature operator.}
We claim that it holds
\begin{align}
\nonumber
&- \int_0^T\int_{\Omega} \eps\nabla\varphi_\eps \cdot (\nabla\varphi_\eps \cdot \nabla) \eta \,dx dt
\\& \label{eq:approxGibbsThomson}
= - \int_0^T\int_{\Omega} H_\eps(\eta\cdot\no_\eps)|\nabla\varphi_\eps| \,dx dt
- \int_0^T\int_{\Omega} (\nabla\cdot \eta) |\nabla\psi_\eps| \,dx dt 
\\&~~~ \nonumber
- \int_0^T\int_{\Omega} (\nabla\cdot \eta) \frac{1}{2}\Big(\sqrt{\eps}|\nabla\varphi_\eps|
-\frac{1}{\sqrt{\eps}}\sqrt{2W(\varphi_\eps)}\Big)^2 \,dx dt
\end{align}
for a.e.\ $T \in (0,\Tstrong)$ and all $\eta \in C^\infty_c(\Omega{\times}[0,T];\Rd)$. 
In particular, we record that
\begin{align}
\label{eq:approxGibbsThomsonDivFree}
- \int_0^T\int_{\Omega} \eps\nabla\varphi_\eps \cdot (\nabla\varphi_\eps \cdot \nabla) \eta \,dx dt
= -\int_0^T\int_{\Omega} H_\eps(\eta\cdot\no_\eps)|\nabla\varphi_\eps| \,dx dt
\end{align}
for a.e.\ $T \in (0,\Tstrong)$ and all $\eta \in C^\infty_c(\Omega{\times}[0,T];\Rd)$
with $\nabla\cdot\eta=0$.

Indeed, as we may compute by an integration by parts and~\eqref{eq:identitiesNormals}
\begin{align*}
&- \int_0^T\int_{\Omega} \eps\nabla\varphi_\eps \cdot (\nabla\varphi_\eps \cdot \nabla) \eta \,dx dt
\\&
= \int_0^T\int_{\Omega} \eps\Delta\varphi_\eps (\eta\cdot\no_\eps)|\nabla\varphi_\eps| \,dx dt
+ \int_0^T\int_{\Omega} \eps \nabla\varphi_\eps\otimes \eta:\nabla^2\varphi_\eps \,dx dt
\\&
= \int_0^T\int_{\Omega} \eps\Delta\varphi_\eps (\eta\cdot\no_\eps)|\nabla\varphi_\eps| \,dx dt
-\int_0^T\int_{\Omega} \frac 12 (\nabla\cdot \eta)\eps|\nabla\varphi_\eps|^2 \,dx dt
\end{align*}
it follows in combination with completing a square that 
\begin{align*}
&- \int_0^T\int_{\Omega} \eps\nabla\varphi_\eps \cdot (\nabla\varphi_\eps \cdot \nabla) \eta \,dx dt
\\&
= \int_0^T\int_{\Omega} \eps\Delta\varphi_\eps (\eta\cdot\no_\eps)|\nabla\varphi_\eps| \,dx dt
+ \int_0^T\int_{\Omega} (\nabla\cdot \eta)\frac{1}{\eps}W(\varphi_\eps) \,dx dt
\\&~~~
- \int_0^T\int_{\Omega} (\nabla\cdot \eta) |\nabla\psi_\eps| \,dx dt 
- \int_0^T\int_{\Omega} (\nabla\cdot \eta) \frac{1}{2}\Big(\sqrt{\eps}|\nabla\varphi_\eps|
-\frac{1}{\sqrt{\eps}}\sqrt{2W(\varphi_\eps)}\Big)^2 \,dx dt
\end{align*}
for a.e.\ $T \in (0,\Tstrong)$ and all $\eta \in C^\infty_c(\Omega{\times}[0,T];\Rd)$.
Further integrating by parts in the second term on the
right hand side of the last display, we obtain the claim~\eqref{eq:approxGibbsThomson} 
from recalling the definition~\eqref{eq:phaseFieldCurvature} of~$H_\eps$.

\textit{Step 5: Exploiting control on the time evolution of~$\xi$.}
We next claim that
\begin{align}
\label{eq:auxRelEntropyTimeEvol10}
\mathrm{Res}^{(1)}(T) &=
- \int_0^T\int_{\Omega} \big(\partial_t\xi + (B\cdot\nabla)\xi 
+ (\nabla B)^\mathsf{T}\xi\big)\cdot(\no_\eps-\xi)|\nabla\psi_\eps| \,dx dt
\\&~~~ \nonumber
- \int_0^T\int_{\Omega} \xi\cdot\big(\partial_t {+} (B\cdot\nabla)\big)\xi |\nabla\psi_\eps| \,dx dt
\\&~~~ \nonumber
- \int_0^T\int_{\Omega} \nabla B:(\xi-\no_\eps)\otimes(\xi-\no_\eps) |\nabla\psi_\eps| \,dx dt
\\&~~~ \nonumber
+ \int_0^T\int_{\Omega} (\nabla\cdot B)(1-\xi\cdot\no_\eps)|\nabla\psi_\eps| \,dx dt
\\&~~~ \nonumber
+ \int_0^T\int_{\Omega} (\nabla\cdot B) \frac{1}{2}\Big(\sqrt{\eps}|\nabla\varphi_\eps|
-\frac{1}{\sqrt{\eps}}\sqrt{2W(\varphi_\eps)}\Big)^2 \,dx dt
\\&~~~ \nonumber
- \int_0^T\int_{\Omega} \big(\no_\eps\otimes\no_\eps - \xi\otimes\xi\big) : \nabla B\, 
\big(\eps|\nabla\varphi_\eps|^2 - |\nabla\psi_\eps|\big) \,dx dt
\\&~~~ \nonumber
- \int_0^T\int_{\Omega} \xi\otimes\xi : \nabla B\, 
\big(\eps|\nabla\varphi_\eps|^2 - |\nabla\psi_\eps|\big) \,dx dt
\\&~~~ \nonumber
+ \mathrm{Res}^{(2)}(T)
\end{align}
for a.e.\ $T \in (0,\Tstrong)$, where
\begin{align}
\label{eq:auxRelEntropyTimeEvol11}
\mathrm{Res}^{(2)}(T) &:=
- \int_0^T\int_{\Omega} \frac{1}{\eps}|H_\eps|^2 \,dxdt
- \int_0^T\int_{\Omega} (\nabla\cdot\xi)\frac{H_\eps}{\sqrt{\eps}}
	\frac{\sqrt{2W(\varphi_\eps)}}{\sqrt{\eps}}  \,dx dt
\\&~~~ \nonumber
- \int_0^T\int_{\Omega} c_0\chi ((v_\eps {-} v)\cdot\nabla)(\nabla\cdot\xi) \,dx dt
\\&~~~ \nonumber
- \int_0^T\int_{\Omega} (\nabla\cdot\xi)((v_\eps {-} B)\cdot\no_\eps)|\nabla\psi_\eps| \,dx dt
\\&~~~ \nonumber
- \int_0^T\int_{\Omega} H_\eps\big((v {-} B)\cdot\no_\eps\big)|\nabla\varphi_\eps| \,dx dt.
\end{align}
For a proof of~\eqref{eq:auxRelEntropyTimeEvol10} and~\eqref{eq:auxRelEntropyTimeEvol11}, 
we first record that it follows from plugging in 
$\eta=v$ as a test function in~\eqref{eq:approxGibbsThomsonDivFree},
which is indeed an admissible choice after a standard mollification argument, that
\begin{align}
\label{eq:auxRelEntropyTimeEvol12}
- \int_0^T\int_{\Omega} \eps\nabla\varphi_\eps \cdot (\nabla\varphi_\eps \cdot \nabla) v \,dx dt
= - \int_0^T\int_{\Omega} H_\eps(v\cdot\no_\eps)|\nabla\varphi_\eps| \,dx dt
\end{align}
for a.e.\ $T \in (0,\Tstrong)$. The next step in the computation takes care of the term involving
the time derivative of the vector field $\xi$. Adding and subtracting 
the anticipated PDE for the time evolution
of the vector field $\xi$, cf.\ the left hand side of~\eqref{eq:conditionsXiB6},
and analogously for the anticipated equation for the
time evolution of its length, cf.\ the left hand side of~\eqref{eq:conditionsXiB7},
we get making also use of~\eqref{eq:identitiesNormals}
\begin{align}
\nonumber
&- \int_0^T\int_{\Omega} \nabla\psi_\eps \cdot\partial_t\xi \,dx dt 
\\& \nonumber
= - \int_0^T\int_{\Omega} \big(\partial_t\xi {+} (B\cdot\nabla)\xi 
{+} (\nabla B)^\mathsf{T}\xi\big)\cdot\no_\eps|\nabla\psi_\eps| \,dx dt
\\&~~~ \nonumber
+ \int_0^T\int_{\Omega} \big((B\cdot\nabla)\xi + (\nabla B)^\mathsf{T}\xi\big)
\cdot\no_\eps|\nabla\psi_\eps| \,dx dt
\\& \label{eq:auxRelEntropyTimeEvol13}
= - \int_0^T\int_{\Omega} \big(\partial_t\xi {+} (B\cdot\nabla)\xi {+} (\nabla B)^\mathsf{T}\xi\big)
\cdot(\no_\eps {-} \xi)|\nabla\psi_\eps| \,dx dt 
\\&~~~ \nonumber
- \int_0^T\int_{\Omega} \xi\cdot\big(\partial_t {+} (B\cdot\nabla)\big)\xi |\nabla\psi_\eps| \,dx dt 
\\&~~~ \nonumber
- \int_0^T\int_{\Omega} \nabla B:\xi\otimes(\xi-\no_\eps) |\nabla\psi_\eps| \,dx dt
\\&~~~ \nonumber
+ \int_0^T\int_{\Omega} \big(\no_\eps \cdot (B\cdot\nabla)\xi\big) |\nabla\psi_\eps| \,dx dt
\end{align}
for a.e.\ $T \in (0,\Tstrong)$. We further rewrite the last term 
in the preceding identity by repeatedly adding zero and also using that
$(B\cdot\nabla)\xi = \nabla\cdot(\xi\otimes B)-(\nabla\cdot B)\xi$ as follows
\begin{align*}
&\int_0^T\int_{\Omega} \big(\no_\eps \cdot (B\cdot\nabla)\xi\big) |\nabla\psi_\eps| \,dx dt
\\&
= \int_0^T\int_{\Omega} (\nabla\cdot B)(1-\xi\cdot\no_\eps)|\nabla\psi_\eps| \,dx dt
- \int_0^T\int_{\Omega} (\mathrm{Id}-\no_\eps\otimes\no_\eps):\nabla B\, |\nabla\psi_\eps| \,dx dt
\\&~~~
+ \int_0^T\int_{\Omega} \big(\nabla\cdot(\xi\otimes B)\big)\cdot\no_\eps |\nabla\psi_\eps| \,dx dt 
- \int_0^T\int_{\Omega} \no_\eps\otimes\no_\eps:\nabla B\,|\nabla\psi_\eps| \,dx dt. 
\end{align*}
In order to combine the last two terms, we compute based on 
a standard mollifier approximation (w.r.t.\ the spatial variable)~$\xi_k:=
\theta_k\ast\xi$ and $B_k := \theta_k\ast B$ 
for the vector fields~$\xi$ and~$B$, respectively (recall the 
regularity assumptions~\eqref{eq:conditionsXiB1}--\eqref{eq:conditionsXiB3}), 
another integration by parts as well as~\eqref{eq:identitiesNormals} that
\begin{align*}
&\int_0^T\int_{\Omega} \big(\nabla\cdot(\xi\otimes B)\big)\cdot\no_\eps |\nabla\psi_\eps| \,dx dt
\\&
= - \lim_{k\to\infty} \int_0^T\int_{\Omega} 
\psi_\eps\nabla\cdot\big(\nabla\cdot(\xi_k \otimes B_k)\big) \,dx dt
\\&
=- \lim_{k\to\infty} \int_0^T\int_{\Omega} 
\psi_\eps\nabla\cdot\big(\nabla\cdot(B_k \otimes \xi_k)\big) \,dx dt
\\&
= \int_0^T\int_{\Omega} (\nabla\cdot\xi)B\cdot\no_\eps |\nabla\psi_\eps| \,dx dt
+ \int_0^T\int_{\Omega} \no_\eps\otimes\xi:\nabla B \,|\nabla\psi_\eps| \,dx dt.
\end{align*}
All in all, feeding back into~\eqref{eq:auxRelEntropyTimeEvol13} the previous two displays 
we consequently obtain
\begin{align}
\nonumber
&- \int_0^T\int_{\Omega} \nabla\psi_\eps \cdot\partial_t\xi \,dx dt 
\\& \label{eq:auxRelEntropyTimeEvol14}
= - \int_0^T\int_{\Omega} \big(\partial_t\xi {+} (B\cdot\nabla)\xi {+} (\nabla B)^\mathsf{T}\xi\big)
\cdot(\no_\eps {-} \xi)|\nabla\psi_\eps| \,dx dt 
\\&~~~ \nonumber
- \int_0^T\int_{\Omega} \xi\cdot\big(\partial_t {+} (B\cdot\nabla)\big)\xi |\nabla\psi_\eps| \,dx dt 
\\&~~~ \nonumber
- \int_0^T\int_{\Omega} \nabla B:(\xi-\no_\eps)\otimes(\xi-\no_\eps) |\nabla\psi_\eps| \,dx dt
\\&~~~ \nonumber
+ \int_0^T\int_{\Omega} (\nabla\cdot B)(1-\xi\cdot\no_\eps)|\nabla\psi_\eps| \,dx dt
\\&~~~ \nonumber
+ \int_0^T\int_{\Omega} (\nabla\cdot\xi)B\cdot\no_\eps |\nabla\psi_\eps| \,dx dt
\\&~~~ \nonumber
- \int_0^T\int_{\Omega} (\mathrm{Id}-\no_\eps\otimes\no_\eps):\nabla B\, |\nabla\psi_\eps| \,dx dt
\end{align}
for a.e.\ $T \in (0,\Tstrong)$. Finally
appealing to~\eqref{eq:approxGibbsThomson} with the admissible
choice of the test vector field~$\eta=B$ (again after a 
standard mollification argument) in form of
\begin{align}
\nonumber
&- \int_0^T\int_{\Omega} (\mathrm{Id}-\no_\eps\otimes\no_\eps):\nabla B \, |\nabla\psi_\eps| \,dx dt
\\ \label{eq:auxRelEntropyTimeEvol15}
&= \int_0^T\int_{\Omega} H_\eps (B\cdot\no_\eps)|\nabla\varphi_\eps|
\\&~~~ \nonumber
+ \int_0^T\int_{\Omega} (\nabla\cdot B) \frac{1}{2}\Big(\sqrt{\eps}|\nabla\varphi_\eps| \,dx dt
- \frac{1}{\sqrt{\eps}}\sqrt{2W(\varphi_\eps)}\Big)^2 
\\&~~~ \nonumber
- \int_0^T\int_{\Omega} \big(\no_\eps\otimes\no_\eps - \xi\otimes\xi\big) : \nabla B\, 
\big(\eps|\nabla\varphi_\eps|^2 - |\nabla\psi_\eps|\big) \,dx dt
\\&~~~ \nonumber
- \int_0^T\int_{\Omega} \xi\otimes\xi : \nabla B\, 
\big(\eps|\nabla\varphi_\eps|^2 - |\nabla\psi_\eps|\big) \,dx dt 
\end{align}
for a.e.\ $T \in (0,\Tstrong)$, the claims~\eqref{eq:auxRelEntropyTimeEvol10} 
and~\eqref{eq:auxRelEntropyTimeEvol11} therefore follow from the
combination of the identities~\eqref{eq:auxRelEntropyTimeEvol12},
\eqref{eq:auxRelEntropyTimeEvol14} and~\eqref{eq:auxRelEntropyTimeEvol15}.

\textit{Step 6: Generating dissipation squares.} For the final
step, we claim that
\begin{align}
\label{eq:auxRelEntropyTimeEvol16}
\mathrm{Res}^{(2)} &\leq - \int_{0}^{T}\int_{\Omega} \frac{1}{2\eps} 
	\Big| H_\eps + \sqrt{2W(\varphi_\eps)}(\nabla\cdot\xi) \Big|^2 \,dx dt
\\&~~~ \nonumber
- \int_{0}^{T}\int_{\Omega} \frac{1}{2\eps} 
  \Big| H_\eps - \big((B - v)\cdot\xi\big)\eps|\nabla\varphi_\eps| \Big|^2 \,dx dt
\\&~~~ \nonumber
- \int_{0}^{T}\int_{\Omega} (c_0\chi - \psi_\eps) \big((v_\eps-v)\cdot\nabla\big)(\nabla\cdot\xi) \,dx dt
\\&~~~ \nonumber
+ \int_{0}^{T}\int_{\Omega} \big|(B-v)\cdot\xi + \nabla\cdot\xi\big|^2\eps|\nabla\varphi_\eps|^2 \,dx dt
\\&~~~ \nonumber
+ \int_{0}^{T}\int_{\Omega} |\nabla\cdot\xi|^2
\bigg(\frac{\sqrt{2W(\varphi_\eps)}}{\sqrt{\eps}}-\sqrt{\eps}|\nabla\varphi_\eps|\bigg)^2 \,dx dt
\\&~~~ \nonumber
- \int_{0}^{T}\int_{\Omega} \frac{1}{\sqrt{\eps}}\big(H_\eps + \sqrt{2W(\varphi_\eps)}(\nabla\cdot\xi)\big)
\big((v-B)\cdot(\no_\eps-\xi)\big)\sqrt{\eps}|\nabla\varphi_\eps| \,dx dt
\end{align}
for a.e.\ $T \in (0,\Tstrong)$. Indeed, once we 
established~\eqref{eq:auxRelEntropyTimeEvol16} the 
asserted relative energy inequality from Proposition~\ref{prop:relEntropyInequality}
is a consequence of collecting the estimates and identities
from~\eqref{eq:auxRelEntropyTimeEvol9}, \eqref{eq:auxRelEntropyTimeEvol10} 
and~\eqref{eq:auxRelEntropyTimeEvol11}, respectively.

For a proof of~\eqref{eq:auxRelEntropyTimeEvol16},
we start by spending half of the available interfacial 
energy dissipation in form of 
\begin{align}
\nonumber
&- \int_0^T\int_{\Omega} \frac{1}{\eps}|H_\eps|^2 \,dxdt
- \int_0^T\int_{\Omega} (\nabla\cdot\xi)\frac{H_\eps}{\sqrt{\eps}}
	\frac{\sqrt{2W(\varphi_\eps)}}{\sqrt{\eps}}  \,dx dt
\\& \label{eq:auxRelEntropyTimeEvol17}
= - \int_{0}^{T}\int_{\Omega} \frac{1}{2\eps} 
	\Big| H_\eps + \sqrt{2W(\varphi_\eps)}(\nabla\cdot\xi) \Big|^2 \,dx dt
\\&~~~ \nonumber
- \int_0^T\int_{\Omega} \frac{1}{2\eps}|H_\eps|^2 \,dxdt
+ \int_0^T\int_{\Omega} \frac{1}{2}|\nabla\cdot\xi|^2 \Big(
\frac{\sqrt{2W(\varphi_\eps)}}{\sqrt{\eps}}\Big)^2 \,dxdt.
\end{align}
We proceed by combining the remaining terms. Adding zero, recalling~\eqref{eq:identitiesNormals},
and integrating by parts (using in the process the solenoidality of
the velocity fields) results in
\begin{align}
\nonumber
&- \int_{0}^{T}\int_{\Omega} c_0\chi \big((v_\eps-v)\cdot\nabla\big)(\nabla\cdot\xi) \,dxdt
- \int_{0}^{T}\int_{\Omega} (\nabla\cdot\xi)\big((v_\eps-B)\cdot\no_\eps\big)|\nabla\psi_\eps| \,dxdt
\\& \label{eq:auxRelEntropyTimeEvol18}
= - \int_{0}^{T}\int_{\Omega} (c_0\chi - \psi_\eps) \big((v_\eps-v)\cdot\nabla\big)(\nabla\cdot\xi) \,dxdt
\\&~~~ \nonumber
- \int_{0}^{T}\int_{\Omega} (\nabla\cdot\xi)\big((v-B)\cdot(\no_\eps-\xi)\big)|\nabla\psi_\eps| \,dxdt
\\&~~~ \nonumber
- \int_{0}^{T}\int_{\Omega} (\nabla\cdot\xi)\big((v-B)\cdot\xi\big)|\nabla\psi_\eps| \,dxdt.
\end{align}
Moreover, adding zero as well as spending the remaining half of the available interfacial 
energy dissipation leads to
\begin{align}
\nonumber
&- \int_{0}^{T}\int_{\Omega} \frac{1}{2\eps} |H_\eps|^2 \,dxdt
- \int_{0}^{T}\int_{\Omega} H_\eps \big((v-B)\cdot\no_\eps\big)|\nabla\varphi_\eps| \,dxdt
\\& \label{eq:auxRelEntropyTimeEvol19}
= - \int_{0}^{T}\int_{\Omega} \frac{1}{2\eps} 
\Big| H_\eps - \big((B - v)\cdot\xi\big)\eps|\nabla\varphi_\eps| \Big|^2 \,dxdt
\\&~~~ \nonumber
+ \int_{0}^{T}\int_{\Omega} \frac{1}{2}|(B-v)\cdot\xi|^2\eps|\nabla\varphi_\eps|^2 \,dxdt
\\&~~~ \nonumber
- \int_{0}^{T}\int_{\Omega} H_\eps \big((v-B)\cdot(\no_\eps-\xi)\big)|\nabla\varphi_\eps| \,dxdt.
\end{align}
Two simple steps remain until we reach the desired final form~\eqref{eq:auxRelEntropyTimeEvol16}. 
The first makes use of $|\nabla\psi_\eps|=\sqrt{2W(\varphi_\eps)}|\nabla\varphi_\eps|$ 
which in turn allows to combine
\begin{align}
\nonumber
&- \int_{0}^{T}\int_{\Omega} H_\eps \big((v-B)\cdot(\no_\eps-\xi)\big)|\nabla\varphi_\eps| \,dxdt
\\&~~~ \nonumber
- \int_{0}^{T}\int_{\Omega} (\nabla\cdot\xi)\big((v-B)\cdot(\no_\eps-\xi)\big)|\nabla\psi_\eps| \,dxdt
\\& \label{eq:auxRelEntropyTimeEvol20}
= - \int_{0}^{T}\int_{\Omega} \frac{1}{\sqrt{\eps}}
\big(H_\eps + \sqrt{2W(\varphi_\eps)}(\nabla\cdot\xi)\big)
\big((v-B)\cdot(\no_\eps-\xi)\big)\sqrt{\eps}|\nabla\varphi_\eps| \,dxdt.
\end{align}
The second consists of collecting the square, adding zero
and exploiting the simple estimate $(a+b)^2\leq 2(a^2+b^2)$ in form of
\begin{align}
\nonumber
&\int_{0}^{T}\int_{\Omega} \frac{1}{2}|(B-v)\cdot\xi|^2\eps|\nabla\varphi_\eps|^2 \,dxdt
+ \int_{0}^{T}\int_{\Omega} \frac{1}{2\eps} \big| \sqrt{2W(\varphi_\eps)}(\nabla\cdot\xi) \big|^2 \,dxdt
\\&~~~ \nonumber
+ \int_{0}^{T}\int_{\Omega} (\nabla\cdot\xi)\big((B-v)\cdot\xi\big)|\nabla\psi_\eps| \,dxdt
\\& \nonumber
= \int_{0}^{T}\int_{\Omega} \frac{1}{2\eps}\Big|\big((B-v)\cdot\xi\big)\eps|\nabla\varphi_\eps|
+ \sqrt{2W(\varphi_\eps)}(\nabla\cdot\xi)\Big|^2 \,dxdt
\\& \nonumber
= \int_{0}^{T}\int_{\Omega} \frac{1}{2\eps}
\Big|\big((B {-} v)\cdot\xi {+} \nabla\cdot\xi\big)\eps|\nabla\varphi_\eps| \,dxdt
+ \big(\sqrt{2W(\varphi_\eps)} {-} \eps|\nabla\varphi_\eps|\big)(\nabla\cdot\xi)\Big|^2 \,dxdt
\\& \label{eq:auxRelEntropyTimeEvol21}
\leq \int_{0}^{T}\int_{\Omega} \big|(B-v)\cdot\xi + \nabla\cdot\xi\big|^2\eps|\nabla\varphi_\eps|^2 \,dx dt
\\&~~~ \nonumber
+ \int_{0}^{T}\int_{\Omega} |\nabla\cdot\xi|^2
\bigg(\frac{\sqrt{2W(\varphi_\eps)}}{\sqrt{\eps}}-\sqrt{\eps}|\nabla\varphi_\eps|\bigg)^2 \,dx dt.
\end{align}
Hence, the combination of~\eqref{eq:auxRelEntropyTimeEvol17}--\eqref{eq:auxRelEntropyTimeEvol21}
implies the claim~\eqref{eq:auxRelEntropyTimeEvol16} and thus concludes the 
derivation of the relative energy inequality.
\end{proof}

\section{Time evolution of the error in the phase indicators}\label{sec:timeEvolWeight}

\begin{lemma}
\label{lem:evolBulkError}
Let the assumptions and notation of Theorem~\ref{theo:mainResult}
be in place, let~$(\xi,B)$ be a pair of vector fields
subject to~\eqref{eq:conditionsXiB1} and~\eqref{eq:conditionsXiB2}, respectively,
let~$\vartheta$ be a time-dependent weight satisfying~\eqref{eq:conditionsWeight1}
as well as~\eqref{eq:conditionsWeight4}--\eqref{eq:conditionsWeight5},
and recall finally the notation~\eqref{eq_defWeakNormal} and~\eqref{eq:phaseFieldCurvature}.
The time evolution of the error in the phase indicators~\eqref{eq:defWeightVol} 
may then be represented as follows: it holds
\begin{align*}
&E_{\mathrm{vol}}[\varphi_\eps|\chi](T)
\\&
= E_{\mathrm{vol}}[\varphi_\eps|\chi](0) 
+ \int_{0}^{T}\int_{\Omega} (\psi_\eps - c_0\chi)
\big(\partial_t \vartheta + (B\cdot\nabla)\vartheta\big) \,dx dt
\\&~~~
+ \int_{0}^{T}\int_{\Omega} (\psi_\eps - c_0\chi) \vartheta \nabla\cdot B \,dx dt
\\&~~~
+ \int_{0}^{T}\int_{\Omega} \vartheta\big((B-v)\cdot(\no_\eps-\xi)\big)|\nabla\psi_\eps| \,dx dt
\\&~~~
- \int_{0}^{T}\int_{\Omega} (\psi_\eps - c_0\chi)\big((v-v_\eps)\cdot\nabla)\vartheta \,dx dt
\\&~~~
+ \int_0^T\int_{\Omega} \vartheta\big((B-v)\cdot \xi\big)
\big(|\nabla\psi_\eps| - \eps|\nabla\varphi_\eps|^2\big) \,dx dt
\\&~~~
+ \int_0^T\int_{\Omega} \Big(\big((B-v)\cdot\xi\big)\sqrt{\eps}|\nabla\varphi_\eps| 
- \frac{H_\eps}{\sqrt{\eps}}\Big)\vartheta\sqrt{\eps}|\nabla\varphi_\eps| \,dx dt
\\&~~~
+ \int_0^T\int_{\Omega} \vartheta \Big(\frac{H_\eps}{\sqrt{\eps}}
+ \frac{\sqrt{2W(\varphi_\eps)}}{\sqrt{\eps}}(\nabla\cdot\xi)\Big)
\Big(\sqrt{\eps}|\nabla\varphi_\eps| - \frac{\sqrt{2W(\varphi_\eps)}}{\sqrt{\eps}}\Big) \,dx dt
\\&~~~
+ \int_0^T\int_{\Omega} \vartheta(\nabla\cdot\xi)\Big|
\sqrt{\eps}|\nabla\varphi_\eps| - \frac{\sqrt{2W(\varphi_\eps)}}{\sqrt{\eps}}\Big|^2 \,dx dt
\\&~~~
- \int_0^T\int_{\Omega} \vartheta\sqrt{\eps}|\nabla\varphi_\eps|
\Big(\sqrt{\eps}|\nabla\varphi_\eps| - \frac{\sqrt{2W(\varphi_\eps)}}{\sqrt{\eps}}\Big) \,dx dt
\end{align*}
for a.e.\ $T \in (0,\Tstrong)$.
\end{lemma}

\begin{proof}
First, since for a.e.\ $t \in (0,\Tweak)$ we have $|\varphi_\eps(\cdot,t)| \leq 1$ 
$\mathcal{L}^d$-a.e.\ in~$\Omega$, which in turn is guaranteed by a maximum principle
argument and the analogous condition for the initial phase-field~$\varphi_{\eps,0}$,
it follows that $\psi_\eps \in [0,c_0]$ and thus from the conditions~\eqref{eq:conditionsWeight4} 
and~\eqref{eq:conditionsWeight5} that
\begin{align}
\label{eq:auxBulkErrorTimeEvol1}
E_{\mathrm{vol}}[\varphi_\eps|\chi](T)
= \int_{\Omega} (\psi_\eps {-} c_0\chi)(\cdot,T) \, \vartheta(\cdot,T) \,dx
\end{align}
for a.e.\ $T \in (0,\Tstrong)$.
Second, as a consequence of $\vartheta \equiv 0$ along~$\II$, cf.\ again
to this end~\eqref{eq:conditionsWeight4} and~\eqref{eq:conditionsWeight5}, 
we also observe that it holds $c_0\int_0^T\int_{\Omega} \vartheta\partial_t\chi \,dx dt = 0$
for a.e.\ $T \in (0,\Tstrong)$.
Hence, based on these two properties we may compute
by an application of the fundamental theorem of calculus,
the regularity of the phase-field~$\varphi_\eps$, 
the condition~\eqref{eq:conditionsWeakSolPhaseField7},
adding zero, as well as the product rule that
\begin{align}
\nonumber
E_{\mathrm{vol}}[\varphi_\eps|\chi](T)
&= E_{\mathrm{vol}}[\varphi_\eps|\chi](0)
+ \int_0^T\int_{\Omega} (\psi_\eps {-} c_0\chi) \partial_t\vartheta \,dx dt
+ \int_0^T\int_{\Omega} \vartheta \partial_t\psi_\eps  \,dx dt
\\& \label{eq:auxBulkErrorTimeEvol2}
= E_{\mathrm{vol}}[\varphi_\eps|\chi](0)
+ \int_0^T\int_{\Omega} (\psi_\eps {-} c_0\chi) 
(\partial_t\vartheta {+} (B \cdot \nabla)\vartheta) \,dx dt
\\&~~~ \nonumber
- \int_0^T\int_{\Omega} (\psi_\eps {-} c_0\chi)\nabla\cdot \big(B\vartheta\big) \,dx dt
+ \int_0^T\int_{\Omega} (\psi_\eps {-} c_0\chi)\vartheta\nabla\cdot B \,dx dt
\\&~~~ \nonumber
+ \int_0^T\int_{\Omega} \vartheta \partial_t\psi_\eps  \,dx dt
\end{align} 
for a.e.\ $T \in (0,\Tstrong)$. Integrating by parts and exploiting again that
$\vartheta \equiv 0$ along~$\II$, we may upgrade~\eqref{eq:auxBulkErrorTimeEvol2} 
to 
\begin{align}
E_{\mathrm{vol}}[\varphi_\eps|\chi](T)
&= 
\label{eq:auxBulkErrorTimeEvol3}
E_{\mathrm{vol}}[\varphi_\eps|\chi](0)
+ \int_0^T\int_{\Omega} (\psi_\eps {-} c_0\chi) 
(\partial_t\vartheta {+} (B \cdot \nabla)\vartheta) \,dx dt
\\&~~~ \nonumber
+ \int_0^T\int_{\Omega} (\psi_\eps {-} c_0\chi)\vartheta\nabla\cdot B \,dx dt
\\&~~~ \nonumber
+ \int_0^T\int_{\Omega} \vartheta (\partial_t\psi_\eps {+} (B\cdot\nabla)\psi_\eps)  \,dx dt
\end{align} 
for a.e.\ $T \in (0,\Tstrong)$. The remainder of the proof takes care of 
suitably post-processing the last right hand side term from the previous 
display incorporating an advective derivative for the $BV$-approximation~$\psi_\eps$
of~$c_0\chi$. 

To this end, inserting in a first step the identity~\eqref{eq:evolPsi},
making use of~\eqref{eq:identitiesNormals} in a second step, and finally adding zero yields
\begin{align}
\nonumber
&\int_0^T\int_{\Omega} \vartheta (\partial_t\psi_\eps {+} (B\cdot\nabla)\psi_\eps)  \,dx dt
\\& \nonumber
= \int_0^T\int_{\Omega} \vartheta(B\cdot\nabla)\psi_\eps \,dx dt
- \int_0^T\int_{\Omega} \vartheta\Big(\frac{H_\eps}{\sqrt{\eps}} 
\frac{\sqrt{2W(\varphi_\eps)}}{\sqrt{\eps}} + (v_\eps\cdot\nabla)\psi_\eps\Big) \,dx dt
\\& \label{eq:auxBulkErrorTimeEvol4}
= \int_0^T\int_{\Omega} \vartheta\big((B-v)\cdot(\no_\eps-\xi)\big)|\nabla\psi_\eps| \,dx dt
\\&~~~ \nonumber
+ \int_0^T\int_{\Omega} \vartheta\big((B-v)\cdot \xi\big)|\nabla\psi_\eps| \,dx dt
- \int_0^T\int_{\Omega} \vartheta\frac{H_\eps}{\sqrt{\eps}}\frac{\sqrt{2W(\varphi_\eps)}}{\sqrt{\eps}} \,dx dt 
\\&~~~ \nonumber
+ \int_0^T\int_{\Omega} \vartheta\big((v-v_\eps)\cdot\nabla\big)\psi_\eps \,dx dt.
\end{align}
Note next that we may rewrite
\begin{equation}
\label{eq:auxBulkErrorTimeEvol5}
\begin{aligned}
&\int_0^T\int_{\Omega} \vartheta\big((v-v_\eps)\cdot\nabla\big)\psi_\eps \,dx dt
\\&
= - \int_0^T\int_{\Omega} \psi_\eps\big((v {-} v_\eps)\cdot\nabla)\vartheta \,dx dt
= - \int_0^T\int_{\Omega} (\psi_\eps {-} c_0\chi) \big((v {-} v_\eps)\cdot\nabla)\vartheta \,dx dt.
\end{aligned}
\end{equation}
Indeed, the first equality simply relies on the solenoidality of the velocity 
fields~$v_\eps$ and~$v$, respectively, in combination with an integration by parts. 
The second equality in turn exploits $\vartheta \equiv 0$ along~$\II$
which---by the solenoidality of the velocity fields---allows to smuggle in the same term
but with~$\psi_\eps$ being replaced by~$c_0\chi$. Furthermore, adding zero several times entails
\begin{align}
\nonumber
&\int_0^T\int_{\Omega} \vartheta\big((B-v)\cdot \xi\big)|\nabla\psi_\eps| \,dx dt
- \int_0^T\int_{\Omega} \vartheta\frac{H_\eps}{\sqrt{\eps}}\frac{\sqrt{2W(\varphi_\eps)}}{\sqrt{\eps}} \,dx dt  
\\&\label{eq:auxBulkErrorTimeEvol6}
= \int_0^T\int_{\Omega} \vartheta\big((B-v)\cdot \xi\big)
\big(|\nabla\psi_\eps| - \eps|\nabla\varphi_\eps|^2\big) \,dx dt
\\&~~~\nonumber
+ \int_0^T\int_{\Omega} \Big(\big((B-v)\cdot\xi\big)\sqrt{\eps}|\nabla\varphi_\eps| 
- \frac{H_\eps}{\sqrt{\eps}}\Big)\vartheta\sqrt{\eps}|\nabla\varphi_\eps| \,dx dt
\\&~~~\nonumber
+ \int_0^T\int_{\Omega} \vartheta \Big(\frac{H_\eps}{\sqrt{\eps}}
+ \frac{\sqrt{2W(\varphi_\eps)}}{\sqrt{\eps}}(\nabla\cdot\xi)\Big)
\Big(\sqrt{\eps}|\nabla\varphi_\eps| - \frac{\sqrt{2W(\varphi_\eps)}}{\sqrt{\eps}}\Big) \,dx dt
\\&~~~\nonumber
+ \int_0^T\int_{\Omega} \vartheta(\nabla\cdot\xi)\Big|
\sqrt{\eps}|\nabla\varphi_\eps| - \frac{\sqrt{2W(\varphi_\eps)}}{\sqrt{\eps}}\Big|^2 \,dx dt
\\&~~~\nonumber
- \int_0^T\int_{\Omega} \vartheta\sqrt{\eps}|\nabla\varphi_\eps|
\Big(\sqrt{\eps}|\nabla\varphi_\eps| - \frac{\sqrt{2W(\varphi_\eps)}}{\sqrt{\eps}}\Big) \,dx dt.
\end{align}
Feeding back \eqref{eq:auxBulkErrorTimeEvol4}--\eqref{eq:auxBulkErrorTimeEvol6} 
into~\eqref{eq:auxBulkErrorTimeEvol3} thus yields the desired result.
\end{proof}

\section{Proof of main results}\label{sec:proofMainResult}

\begin{proof}[Proof of Theorem~\ref{theo:mainResult}]
The proof is split into six steps.

\textit{Step 1: Construction of the triple $(\xi,B,\vartheta)$.}
Fix $T \in [0,\Tstrong)$. Due to the assumed regularity of the
evolving geometry underlying the strong solutions~$(\chi,v)$,
cf.\ item~\textit{iii)} of Definition~\ref{def:strongSolSharpInterfaceLimit},
there exists a small scale $\delta=\delta(\chi,T) \in (0,\smash{\frac{1}{2}}]$
such that the signed distance~$s_{\II}(\cdot,t)$ to~$\II(t)$ as well as the
nearest point projection~$P_{\II}(\cdot,t)$ onto~$\II(t)$, $t \in [0,T]$,
are regular in the sense that
\begin{align}
\label{eq:proofMainResultAux1}
s_{\II} &\in (C^1_tC^{1}_x \cap C_tC^{3}_x)\big(\overline{B_{2\delta}(\II)}\big),
\\
\label{eq:proofMainResultAux2}
\big((x,t) \mapsto P_{\II}(x,t) := (x {-} (s_{\II} \nabla s_{\II})(x,t),t)\big) 
&\in (C^1_tC_x {\cap} C_tC^{2}_x)\big(\overline{B_{2\delta}(\II)}\big).
\end{align}
We remark in this context that the signed distance is oriented
by the requirement $\nabla s_{\II} = \no_{\II}$ on~$\II$, and
that for every $r \in (0,1]$ the associated space-time tubular neighborhood~$B_r(\II)$
of~$\II$ is defined by
\begin{align}
\label{eq:proofMainResultAux3}
B_r(\II) := \bigcup_{t \in [0,T]} B_r(\II(t)) \times \{t\}. 
\end{align}
Due to~\eqref{eq:conditionsStrongSolSharpInterfaceLimit9} 
and~\eqref{eq:conditionsStrongSolSharpInterfaceLimit10}, we
may further choose~$\delta$ small enough such that
\begin{align}
\label{eq:proofMainResultAux4}
\overline{B_{2\delta}(\II(t))} \subset \Omega
\quad\text{for all } t \in [0,T]. 
\end{align}

Next, let $\bar\eta\colon\mathbb{R}\to[0,1]$ be a smooth
and even profile with $\supp\bar\eta \subset [-1,1]$ and with
quadratic decay at the origin in the sense of
\begin{align}
\label{eq:proofMainResultAux5}
c_{\bar\eta} r^2 \leq 1 - \bar\eta(r) \leq C_{\bar\eta} r^2
\quad\text{for all } r \in [-1,1],
\end{align}
for some constants $0<c_{\bar\eta}<C_{\bar\eta}<\infty$.
We further choose a smooth and even profile~$\widetilde\eta\colon\mathbb{R}\to[0,1]$
such that $\supp\widetilde\eta\subset [-2,2]$ and
\begin{align}
\label{eq:proofMainResultAux13}
\widetilde\eta \equiv 1 \quad\text{on } [-1,1].
\end{align}
We then define two quadratic cutoffs
\begin{align}
\label{eq:proofMainResultAux6}
\eta_{\II}(x,t) &:= \bar\eta\big(s_{\II}(x,t)/\delta\big),
&&(x,t) \in \Rd{\times}[0,T],
\\
\label{eq:proofMainResultAux14}
\widetilde\eta_{\II}(x,t) &:= \widetilde\eta\big(s_{\II}(x,t)/\delta\big),
&&(x,t) \in \Rd{\times}[0,T].
\end{align}
Note that
\begin{align}
\label{eq:proofMainResultAux12}
\supp\eta_{\II}(\cdot,t) &\subset \overline{B_{\delta}(\II(t))}
&&\text{for all } t \in [0,T],
\\
\label{eq:proofMainResultAux15}
\supp\widetilde\eta_{\II}(\cdot,t) &\subset \overline{B_{2\delta}(\II(t))}
&&\text{for all } t \in [0,T],
\\
\label{eq:proofMainResultAux16}
\widetilde\eta_{\II}(\cdot,t) &\equiv 1 
&&\text{on } \overline{B_{\delta}(\II(t))} \text{ for all } t \in [0,T].
\end{align}

With these ingredients in place, we define the vector fields~$\xi$ and~$B$ by means of
\begin{align}
\label{eq:proofMainResultAux7}
\xi &:= \eta_{\II} \nabla s_{\II}
&&\text{on } \Omega \times [0,T],
\\
\label{eq:proofMainResultAux8}
B &:= \big(\big((v \circ P_{\II}) \cdot \nabla s_{\II}\big) 
- (\Delta s_{\II}) \circ P_{\II}\big) \,\widetilde\eta_{\II}\nabla s_{\II}
&&\text{on } \Omega \times [0,T],
\end{align}
where~$v$ is the velocity field of the strong solution.

For a construction of a suitable weight~$\vartheta$, we 
first fix a smooth and odd map $\bar\vartheta\colon\mathbb{R}\to [-1,1]$
representing a suitable truncation of the (negative of the) identity
in the sense that
\begin{align}
\label{eq:proofMainResultAux9}
&\bar\vartheta \equiv 1 \text{ on } (-\infty,-1]
\quad\text{and}\quad \bar\vartheta \equiv -1 \text{ on } [1,\infty),
\\
\label{eq:proofMainResultAux9b}
&\bar\vartheta > 0 \text{ in } (-1,0)
\quad\text{and}\quad \bar\vartheta < 0 \text{ in } (0,1),
\\
\label{eq:proofMainResultAux10}
&c_{\bar\vartheta}r \leq |\bar\vartheta(r)| \leq C_{\bar\vartheta} r
\quad\text{for all } r \in [-1,1],
\end{align}
for some constants $0<c_{\bar\vartheta}<C_{\bar\vartheta}<\infty$.
We may then define
\begin{align}
\label{eq:proofMainResultAux11}
\vartheta &:= \bar\vartheta\big(s_{\II}/\delta\big)
&&\text{on } \Omega \times [0,T].
\end{align}

\textit{Step 2: Proof of the conditions~\emph{\eqref{eq:conditionsXiB1}--\eqref{eq:conditionsXiB9}}.}
The required regularity~\eqref{eq:conditionsXiB1}--\eqref{eq:conditionsXiB3}
is immediate from the definitions~\eqref{eq:proofMainResultAux7}
and~\eqref{eq:proofMainResultAux8} as well as the 
regularity~\eqref{eq:conditionsStrongSolSharpInterfaceLimit4}, 
\eqref{eq:proofMainResultAux1} and~\eqref{eq:proofMainResultAux2}
of the associated building blocks. Furthermore, because of~\eqref{eq:proofMainResultAux12},
\eqref{eq:proofMainResultAux15} and~\eqref{eq:proofMainResultAux4} it follows that the 
vector fields~$\xi$ and~$B$ are indeed compactly supported within~$\Omega$.

We next note that the coercivity estimate~\eqref{eq:conditionsXiB4}
directly follows from the definition~\eqref{eq:proofMainResultAux7}
and the lower bound from~\eqref{eq:proofMainResultAux5}.
For a proof of the consistency conditions~\eqref{eq:conditionsXiB5},
$\xi = \no_{\II}$ along~$\II$ simply follows from
the definition~\eqref{eq:proofMainResultAux7} as well as $\bar\eta(0) = 0$
whereas $\nabla\cdot\xi = -H_{\II}$ along~$\II$ 
is a consequence of the definition~\eqref{eq:proofMainResultAux7},
$\bar\eta'(0)=0$ as well as the well-known identity $\Delta s_{\II} = -H_{\II}$
along~$\II$.

We proceed with a proof of the (approximate) evolution 
equations~\eqref{eq:conditionsXiB6} and~\eqref{eq:conditionsXiB7}.
The argument is based on the claim
\begin{align}
\label{eq:proofMainResultAux17}
\big(\partial_t s_{\II} + (B \cdot \nabla)s_{\II}\big)(\cdot,t) \equiv 0
\quad\text{on } \overline{B_{\delta}(\II(t))},\, t \in [0,T].
\end{align}
Since $\partial_t s_{\II} = - V_{\II}$ along~$\II$, the 
identity~\eqref{eq:proofMainResultAux17} is indeed valid
as a consequence of the definition~\eqref{eq:proofMainResultAux8},
the property~\eqref{eq:proofMainResultAux16}, item~\textit{iv)}
of Definition~\ref{def:strongSolSharpInterfaceLimit} in form
of $V_{\II} = (B\cdot\nabla)s_{\II}$ along~$\II$, as well as
$\partial_t s_{\II} = (\partial_t s_{\II}) \circ P_{\II}$
(the latter following from a straightforward computation
based on differentiating $s_{\II} \circ P_{\II}\equiv 0$).
Note then that~\eqref{eq:proofMainResultAux17}
together with the chain rule immediately implies
\begin{align}
\label{eq:proofMainResultAux18}
\big(\partial_t \eta_{\II} + (B \cdot \nabla)\eta_{\II}\big)(\cdot,t) \equiv 0
\quad\text{on } \overline{B_{\delta}(\II(t))},\, t \in [0,T].
\end{align}
This in turn directly entails~\eqref{eq:conditionsXiB7} due to
the simple observation $|\xi|^2=\eta_{\II}^2$, which itself follows from
the definition~\eqref{eq:proofMainResultAux7}. Furthermore,
\eqref{eq:proofMainResultAux18} reduces the proof of~\eqref{eq:conditionsXiB6}
to a proof of
\begin{align}
\label{eq:proofMainResultAux19}
\big(\partial_t \nabla s_{\II} {+} (B \cdot \nabla)\nabla s_{\II}
{+} (\nabla B)^\mathsf{T}\nabla s_{\II}\big)(\cdot,t) \equiv 0
\quad\text{on } \overline{B_{\delta}(\II(t))},\, t \in [0,T].
\end{align}
However, \eqref{eq:proofMainResultAux19} simply follows from
taking the spatial gradient of~\eqref{eq:proofMainResultAux17}. 

We finally note that~\eqref{eq:conditionsXiB8} and~\eqref{eq:conditionsXiB9}
are straightforward consequences of the definitions~\eqref{eq:proofMainResultAux7}
and~\eqref{eq:proofMainResultAux8}, the regularity of the associated building blocks,
as well as the already established property~\eqref{eq:conditionsXiB5}.

\textit{Step 3: Proof of the conditions~\emph{\eqref{eq:conditionsWeight1}--\eqref{eq:conditionsWeight6}}.}
The regularity requirements~\eqref{eq:conditionsWeight1} and~\eqref{eq:conditionsWeight2}
are immediate from the definition~\eqref{eq:proofMainResultAux11}
and the regularity~\eqref{eq:proofMainResultAux1}. The coercivity estimate~\eqref{eq:conditionsWeight3}
as well as the sign conditions~\eqref{eq:conditionsWeight4}--\eqref{eq:conditionsWeight5}
in turn follow directly from the definition~\eqref{eq:proofMainResultAux11} and
the properties~\eqref{eq:proofMainResultAux9}--\eqref{eq:proofMainResultAux10}.
Finally, the approximate transport equation~\eqref{eq:conditionsWeight6}
simply results from a combination of~\eqref{eq:proofMainResultAux17} and the chain rule.

\textit{Step 4: Derivation of the stability  
estimates~\eqref{eq:gronwallEstimateRelativeEntropy} and~\eqref{eq:gronwallEstimatePhaseError}.}
Applying \eqref{eq:coercivityRelEntropy1}--\eqref{eq:coercivityRelEntropy4} 
and \eqref{eq:conditionsXiB6}--\eqref{eq:conditionsXiB7}  to the inequality in  
Proposition~\ref{prop:relEntropyInequality}, we have for a.e.\ $T\in (0,T_*)$ that
\begin{align}
&E[\varphi_\eps,v_\eps|\chi,v](T)\nonumber
\\
&\leq E[\varphi_\eps,v_\eps|\chi,v](0)
 +C\int_{0}^{T}E[\varphi_\eps,v_\eps|\chi,v](t) dt- \int_{0}^{T}\int_{\Omega} |\nabla v_\eps - \nabla v|^2 \,dx dt\nonumber
\\&~~~
- \int_{0}^{T}\int_{\Omega} \frac{1}{2\eps} 
	\Big| H_\eps + \sqrt{2W(\varphi_\eps)}(\nabla\cdot\xi) \Big|^2 \,dx dt\label{thm1prooftail1}
\\&~~~
- \int_{0}^{T}\int_{\Omega} \frac{1}{2\eps} 
  \Big| H_\eps - \big((B - v)\cdot\xi\big)\eps|\nabla\varphi_\eps| \Big|^2 \,dx dt\label{thm1prooftail2}
\\&~~~
- \int_{0}^{T}\int_{\Omega} (c_0\chi - \psi_\eps) \big((v_\eps-v)\cdot\nabla\big)(\nabla\cdot\xi) \,dx dt\label{thm1prooftail3}
\\&~~~
+ \int_{0}^{T}\int_{\Omega} \big|(B-v)\cdot\xi + \nabla\cdot\xi\big|^2\eps|\nabla\varphi_\eps|^2 \,dx dt\label{thm1prooftail4}
\\&~~~
- \int_{0}^{T}\int_{\Omega} \frac{1}{\sqrt{\eps}}\big(H_\eps + \sqrt{2W(\varphi_\eps)}(\nabla\cdot\xi)\big)
\big((v-B)\cdot(\no_\eps-\xi)\big)\sqrt{\eps}|\nabla\varphi_\eps| \,dx dt\label{thm1prooftail5}
\\&~~~
-\int_{0}^{T}\int_{\Omega} \big(\no_\eps\otimes\no_\eps - \xi\otimes\xi\big) : \nabla B\, 
\big(\eps|\nabla\varphi_\eps|^2 - |\nabla\psi_\eps|\big) \,dx dt\label{thm1prooftail6}
\\&~~~
-\int_{0}^{T}\int_{\Omega} \xi\otimes\xi : \nabla B\, 
\big(\eps|\nabla\varphi_\eps|^2 - |\nabla\psi_\eps|\big) \,dx dt. \label{thm1prooftail7}
\end{align}
The integral \eqref{thm1prooftail3} is estimated by \eqref{eq:coercivityBulkError2} with a sufficiently small $\lambda>0$. The integral~\eqref{thm1prooftail4} can be estimated using  \eqref{eq:conditionsXiB8} and \eqref{eq:coercivityRelEntropy5}. To estimate  \eqref{thm1prooftail5}, we first employ  Cauchy--Schwarz's inequality,
followed by Young's inequality with a sufficiently small prefactor in order
to exploit the sign of~\eqref{thm1prooftail1}, and then conclude with~\eqref{eq:coercivityRelEntropy5}.
To estimate \eqref{thm1prooftail6}, we write
 \begin{align*}
 \no_\eps\otimes\no_\eps : \nabla B = 
\no_\eps\otimes(\no_\eps {-} \xi): \nabla B
+ (\xi\cdot \nabla )B\cdot (\no_\eps {-} \xi)
+ \xi\otimes\xi : \nabla B
 \end{align*} 
so that we can estimate
 \begin{align*}
 |\no_\eps\otimes\no_\eps : \nabla B - \xi\otimes\xi : \nabla B| \leq  
C\sqrt{1-\no_\eps \cdot \xi}.
 \end{align*}
Because of this,  \eqref{thm1prooftail6} 
can be estimated by using \eqref{eq:coercivityRelEntropy6}. 
Finally, \eqref{thm1prooftail7} can be bounded by
means of~\eqref{eq:conditionsXiB9} and~\eqref{eq:coercivityRelEntropy6}.
All in all, 
\begin{align*}
&E[\varphi_\eps,v_\eps|\chi,v](T)+\frac 12\int_{0}^{T}\int_{\Omega} |\nabla v_\eps - \nabla v|^2 \,dx dt\nonumber
\\&~~~
+ \int_{0}^{T}\int_{\Omega} \frac{1}{4\eps} 
	\Big| H_\eps + \sqrt{2W(\varphi_\eps)}(\nabla\cdot\xi) \Big|^2 \,dx dt \nonumber
\\&~~~
+ \int_{0}^{T}\int_{\Omega} \frac{1}{4\eps} 
  \Big| H_\eps - \big((B - v)\cdot\xi\big)\eps|\nabla\varphi_\eps| \Big|^2 \,dx dt \nonumber
\\&~~~
\leq  E[\varphi_\eps,v_\eps|\chi,v](0)
 +C\int_{0}^{T}E[\varphi_\eps,v_\eps|\chi,v](t)+E_{\mathrm{vol}}[\varphi_\eps|\chi](t) dt.
\end{align*}

Regarding the evolution of the bulk error in Lemma~\ref{lem:evolBulkError}, 
similar considerations based in addition on~\eqref{eq:coercivityRelEntropy5}, 
\eqref{eq:coercivityRelEntropy6}, \eqref{eq:conditionsWeight2},
\eqref{eq:conditionsWeight3}, \eqref{eq:conditionsWeight4}  and~\eqref{eq:conditionsWeight6}
also lead to 
\begin{align*}
&E_{\mathrm{vol}}[\varphi_\eps|\chi](T)
\\&
\leq  E_{\mathrm{vol}}[\varphi_\eps|\chi](0) 
+ C\int_{0}^{T}E[\varphi_\eps,v_\eps|\chi,v](t)+E_{\mathrm{vol}}[\varphi_\eps|\chi](t) dt
\\&~~~
+\int_{0}^{T}\int_{\Omega}\frac 14 |\nabla v_\eps - \nabla v|^2 \,dx dt
\\&~~~
+ \int_0^T\int_{\Omega} \frac 18 \Big(\big((B-v)\cdot\xi\big)\sqrt{\eps}|\nabla\varphi_\eps| 
- \frac{H_\eps}{\sqrt{\eps}}\Big)^2 \,dx dt
\\&~~~
+ \int_0^T\int_{\Omega} \frac 18 \Big(\frac{H_\eps}{\sqrt{\eps}}
+ \frac{\sqrt{2W(\varphi_\eps)}}{\sqrt{\eps}}(\nabla\cdot\xi)\Big)^2 \,dx dt.
\end{align*}
The above two inequalities together with Gr\"{o}nwall's inequality 
therefore allow to conclude the proof.
\end{proof}

\begin{proof}[Proof of Corollary~\ref{cor:sharpConvergenceRates}]
We proceed in two steps.

\textit{Step 1: From~$E_{\mathrm{vol}}$-control to $L^1(\Omega)$-control of the error in
				the phase indicators.}
We claim that for all $T \in [0,\Tstrong)$ there exists 
a constant $C = C(\chi,v,T) \in (0,\infty)$ such that
\begin{align}
\label{eq:unweightedError}
\big\|c_0\chi(\cdot,t) - \psi_\eps(\cdot,t)\big\|_{L^1(\Omega)}^2
\leq CE_{\mathrm{vol}}[\varphi_\eps|\chi](t)
\end{align}
holds true for all $t \in [0,T]$. 
For the simple proof based on a slicing argument
and an application of Fubini's theorem, we refer, e.g., 
to~\cite[Proof of Theorem~1, Step~2]{Fischer2020b}.

\textit{Step 2: Derivation of the sharp convergence rate~\eqref{eq:sharpConvergenceRate}.}
This now immediately follows from post-processing
the stability estimates~\eqref{eq:gronwallEstimateRelativeEntropy} 
and~\eqref{eq:gronwallEstimatePhaseError} by means
of the bound~\eqref{eq:unweightedError} and the assumption~\eqref{eq:wellPreparedInitialData}.
\end{proof}

\appendix
\section{Well-posedness of the sharp interface limit model}
The goal of this part is to construct  a strong solution in the sense of 
Definition~\ref{def:strongSolSharpInterfaceLimit}.
We   consider the following system which is equivalent 
to~\eqref{eq:sharpInterfaceLimit1}--\eqref{eq:sharpInterfaceLimit8} (with $c_0=1$):
\begin{subequations}\label{model3}
\begin{align}
 \partial_t v +(v \cdot \nabla) v &=\Delta v +\nabla \pi&&~\text{in}~\mathring{\O}(t),  \\
   \nabla\cdot v &=0&&~\text{in}~\mathring{\O}(t), \\
 \jump{ 2D v  -\pi~\mathrm{Id} }\cdot\no_{\II}&=H_{\II} \no_{\II}  &&~\text{on}~\II(t), \label{jump stress} \\
 \jump{ v  }&=0 &&~\text{on}~\II(t), \label{jump u}\\
V_{\II} &= \no_{\II}\cdot v  + H_{\II}
&&~\text{on}~\II(t).
\end{align}
\end{subequations}
Here    
$\mathring{\Omega}(t) :=  \O_+ (t)\cup  \O_- (t)$ is the bulk region and     $  D v := \nabla^\mathrm{sym} v $ is 
 the symmetric part of the  flow gradient.     The major difference of \eqref{model3} to the system 
studied  in \cite{KohnePruessWilke} is that in the latter the interface is purely transported by the 
fluid velocity, i.e., $V_{\II} = \no_{\II}\cdot v$.

   \begin{proposition}\label{local strong limit sys}
Let  $p>d+2$ and $\II(0)\subset \O$ be a $C^3$ closed surface. Assume that 
the initial velocity field~$v_0$ satisfies  $v_0 \in W_{p}^{2-2/p}( \mathring{\O}(0)) \cap W^1_{p}(\Omega)$ 
together with the following compatibility conditions:
\begin{align}
\div   v _0=0~&\text{in}~\mathring{\O}(0), \\
 v _0=0~&\text{on}~\partial \Omega,\\
 \jump{ v _0 }=0 ~&\text {on}~ \II(0),\\
\Pi_{\II(0)}\jump{    \nabla^\mathrm{sym} v _0  \cdot\no_{\II}  }=0~&\text{on}~\II(0),
\end{align}
where 
   $\Pi_{\II(t)}=\mathrm{Id}-\no_{\II(t)} \otimes \no_{\II(t)}$  is   the tangential projection.
 Then  there exists $T>0$ such that    \eqref{model3} has  a unique solution $(v ,\pi, \II)$ with  
  \begin{subequations}\label{regularity up}
\begin{align}
&v \in   W_{p}^{1}\left(0, T ; L^{p}(\Omega)\right)  \cap L^{p}\left(0, T ; W_{p}^{2}(\mathring{\O}(t)) \cap W_{p}^{1}(\Omega)\right),  \label{regularity of u}\\
&\pi\in L^{p}\left(0, T ; \mathring{W}_{p}^{1}(\mathring{\O}(t))\right)~\text{with}~\int_\Omega \pi(\cdot,t)=0,
\\
&\jump{ \pi } \in W_{p}^{1-\frac 1p, \frac{1}{2}(1-\frac 1p)}(\II(t)).
\end{align}
\end{subequations} 
Moreover,  the free boundary 
$\II=\cup_{t\in (0,T)}\II(t) {\times} \{t\}$ is parametrized through the diffeomorphism 
$\Theta_{h}(x,t) \colon\mathring{\O}(0)\mapsto \mathring{\O}(t)$ defined by
\begin{align}
\Theta_{h}(x,t):=x+\zeta\left(\frac{\dist(x,\II(0))}\delta\right) h(P_{\II(0)}(x),t)\, \no_{\II(0)} (x) \quad \text { for all } x \in \Rd,\label{classical hanzawa}
\end{align}
where $P_{\II(t)}$ is the nearest point projection to $\II(t)$, $\zeta \in C^{\infty}(\Rd)$   satisfies 
 \begin{equation}
\left|\nabla\zeta(s)\right| \leq 4~\text{in $\Rd$ and }~\zeta(s)=1~\text{for}~ |s| 
\leq 1/2,~\zeta(s)=0~\text{for}~|s| \geq 1,\label{chi def appendix}
\end{equation}
and  $h$ is the height function
\begin{align}\label{regularity h}
h \in   W_p^{2-\frac 1{2p}}\big( J; L^p(\II(0))\big)  
\cap L^p\big( J; W_p^{4-\frac 1p} (\II(0) )  \big)~\text{with}~h|_{t=0}=0.
\end{align}

 \end{proposition}
 
Throughout the next three subsection, we sketch the main ingredients for a proof of
Proposition~\ref{local strong limit sys} and thus conclude with the regularity stated
in Definition~\ref{def:strongSolSharpInterfaceLimit}.

\subsection{Preliminaries}
We will  employ   the following notation:
 \begin{subequations}
 \begin{align} 
\mathring{W}^k_p(\O)&= \{ \partial^\alpha f\in L^p(\O), |\alpha |=k\},\\
W^k_p(\O)&= \{ \partial^\alpha f\in L^p(\O), |\alpha |\leq k\},\\
L^p_{(0)}(\O)&= \Big\{   f\in L^p(\O), \int_\O f\, dx =0 \Big\}.
\end{align}
 \end{subequations}
We need elementary results from interpolation theory and maximal regularity theory 
of parabolic system, see to this end, for instance, \cite[Section 4.10]{MR1345385} and \cite{MR1783238}. 
For a Banach space  $(X,\|\cdot\|_X)$ and $s\in (0,1), p\geq 1$,  we also 
recall the  Sobolev--Slobodeckij space $W_{p}^{s}(J ; X)$ normed by
\begin{align}
\|f\|_{W_{p}^{s}(J; X)}:=\|f\|_{L^p(J ; X)}+[f]_{s,p}<\infty,
\end{align}
where $J=[0,T]$ and 
\begin{align}
[f]_{s,p,X}:=\(\int_{J} \int_{J} \frac{\|f(t)-f(\tau)\|_{X}^p}{|t-\tau|^{p s+1}} d t d \tau\)^{\frac 1 p}.\label{slobo norm}
\end{align}
For $s\in (m,m+1)$ with $m$ being a positive integer, we finally define 
\begin{align}
\|f\|_{W_{p}^{s}(J ; X)}:=\|f\|_{W_p^m(J ; X)}+\max_{|\alpha|=m}[\partial^\alpha f]_{s-m,p,X}.
\end{align}

Now we state the required interpolation inequalities involving these spaces. 
We denote by $(E_0,E_1)_{\theta, p}$  the  real interpolation of Banach 
spaces $E_0$ and $E_1$.  For $\theta\in [0,1]$ and Banach spaces $X_0,X_1$, we then have 
\begin{align}\label{general interpolation}
W_{p}^{1}\left(J ; X_{0}\right) \cap  L_p\left(J ; X_{1}\right)&\hookrightarrow W^{1-\theta}_p\left(J ; (X_{0},X_1)_{\theta,p}\right),\\
W_{p}^{1}\left(J ; X_{0}\right) \cap  L_p\left(J ; X_{1}\right) &\hookrightarrow C^{1-\theta-\frac 1p}\left(J ; (X_0,X_1)_{\theta, p}\right),\label{general interpolation conti}
\end{align}
where the second   embedding holds provided $1-\theta\geq \frac 1p$.
In particular, 
\begin{align}\label{interpo in conti time}
W_{p}^{1}\left(J ; X_{0}\right) \cap L_p\left(J ; X_{1}\right) \hookrightarrow C^{0}\left(J ; (X_0,X_1)_{1- 1/p, p}\right).
\end{align}
We will also rely on the following inequality
\begin{align}\label{wsp embed cs}
\|f\|_{W^s_p(J; X)} \leqslant C_{s, s',p} T^{\left(s'-s\right)+\frac 1p}\|f\|_{C^{s'}(J ; X)}\end{align}
provided that $0<s<s' \leqslant 1,0<T \leqslant 1$.
Indeed, 
\begin{align*}
[f]_{s,p,X}^p & \overset{\eqref{slobo norm}}=\int_J \int_J \frac{\|f(t)-f(\tau)\|_{X}^p}{|t-\tau|^{p s+1}} d t d \tau \\
& ~~\leqslant \int_J \int_J|t-\tau|^{p\left(s'-s\right)-1} d t d \tau\|f\|_{C^{s'}(J ; X)}^p\\
& ~~\leqslant C_{s', s,p} T^{p\left(s'-s\right)+1}\|f\|_{C^{s'}(J ; X)}^p
\end{align*}
for all $0<s<s' \leqslant 1$, and this   implies \eqref{wsp embed cs}.

\subsection{The work of Abels and Moser \cite{zbMATH06951007}}

We shall first reduce \eqref{model3} to a system with fixed domains.
Assume $ \dist(\II, \partial \Omega)>4\delta$.
For $h \in C^{2}(\II(0)\times J)$ with $\|h\|_{L^\infty_{x,t}}<2\delta$,  
the Hanzawa transformation   \eqref{classical hanzawa} is a family of  diffeomorphisms
 \[\Theta_{h}(\cdot,t)\colon \II(0)\mapsto \II(t)~\text{ and }~\mathring{\O}(0)\mapsto \mathring{\O}(t).\]
  Recalling \eqref{chi def appendix}, we see that for a fixed $t$, $\Theta_{h}(\cdot,t)$ is   
	the identity map  on $\Rd \backslash  B_{\delta}(\II(t) )$; in particular near $\partial \Omega .$ Moreover $\operatorname{det} \nabla \Theta_{h} \geq c>0$ and 
\[
  \quad\left\|\nabla \Theta_{h}\right\|_{L^\infty_{x,t}}+\left\|\nabla  \Theta_{h}^{-1}\right\|_{L^\infty_{x,t}} \leq C\left(1+\|h\|_{C^{1}(\II(0)\times J)}\right)
\]
with $c, C>0$ independent of $h$. For technical reasons, it is   simpler to replace $h \circ P_{\II(0)}$ in \eqref{classical hanzawa} by an extension $ E h$ where 
\begin{align}
E\colon X_{0}&:=W_{p}^{1-1/p}(\II(0)) \mapsto W_{p}^{1} \(B_{ \delta}(\II(0)) \)\label{extension operator E}
\end{align}
with  $E$ being  
a bounded and linear extension operator of class
\begin{align}
E \in \mathscr{L}\left(W_{p}^{k-1/p}(\II(0)), W_{p}^{k}\left(B_{ \delta}(\II(0)) \right)\right), \quad \forall k\in \{1,\ldots, 4\}.\label{E extension 123}
\end{align}


Based on this extension operator, we define a modified  Hanzawa transform   by
\begin{align}
y=\tilde{\Theta}_{h}(x,t):=x+\zeta\left(\frac{\dist(x,\II(0))}{\delta}\right) (Eh)(x,t) \no_{\II(0)}(x), \quad \forall x \in \Rd.\label{extend hanzawa}
\end{align}
 Let $(v ,\pi)$ be a solution of \eqref{model3}. We then define the transformed solution by  
\begin{align}\label{tv to v}
\tilde{v }(x, t):=v  (\tilde{\Theta}_{h}(x, t), t ),\quad \tilde{\pi}(x, t):=\pi(\tilde{\Theta}_{h}(x, t), t),
\end{align}
 for $(x, t) \in \Omega \times(0, T)$.
Equivalently   for $y\in \mathring{\O}(t)$, we have  
\begin{align}\label{v to tv}
v (y, t):=\tilde{v } (\tilde{\Theta}_{h}^{-1}(y, t), t ),\quad  \pi(y, t):=\tilde{\pi}(\tilde{\Theta}_{h}^{-1}(y, t), t).
\end{align}
 To identify the system of PDEs satisfied by $(\tilde{v }, \tilde{\pi})$, we introduce  
the following notation:
 \begin{subequations}\label{pertubations}
 \begin{align}
   \nabla^h &:= ((\nabla \tilde{\Theta}_{h})^{-1})^{\mathsf{T}} \nabla,\label{def nablav}\\
   D^h \vv &:= \frac 12 \(\nabla^h \vv+(\nabla^h \vv)^{\mathsf{T}}\),\label{def D h}\\
    \div ^h \tilde{v } &:= \operatorname{tr} \nabla^h \tilde{v },\label{def dh v}\\
 T^h(\tilde{v }, \tilde{\pi}) &:= 2 D^h \tilde{v } -\tilde{\pi}~ \mathrm{Id},\\
 H^h &:= H_{\II}\circ \tilde{\Theta}_h(x,t)\colon \II(0) {\times}  J \rightarrow \Rd,\\
 \no^h &:= \no_{\II}\circ \tilde{\Theta}_h(x,t)\colon \II(0) {\times}  J \rightarrow \Rd.\end{align}
 \end{subequations}
 Then it follows from \eqref{def nablav} and \eqref{v to tv} that 
 \begin{align}\label{gradient equ}
 \nabla_y v|_{y=\tilde{\Theta}_{h}(x, t)} =\nabla^h \tilde{v}(x).
 \end{align}
    Similarly, by \eqref{tv to v} we obtain $\partial_t \tilde{v}(x)=(\partial_t v+  \partial_{t} \tilde{\Theta}_{h}\cdot\nabla v)|_{y=\tilde{\Theta}_{h}(x, t)}.$
With these definitions and formulas, the system \eqref{model3} can be rewritten as one over a fixed domain. 
For simplicity we split  it into two parts, one describing  the hydrodynamics, 
and one for  the evolution of the interface:
\begin{subequations}\label{flat system1}
\begin{align}
\partial_{t} \tilde{v }+\tilde{v } \cdot \nabla^h \tilde{v } &=\div ^h \nabla^h \tilde{v }+\nabla^h \tilde{\pi}+\nabla^h \tilde{v } \,\partial_{t} \tilde{\Theta}_{h} & & \text { in } \mathring{\O}(0) \times(0, T), \\
\div ^h \tilde{v } &=0 & & \text { in } \mathring{\O}(0) \times(0, T), \\
\jump{ 2 D^h \tilde{v } -\tilde{\pi} \mathrm{Id} }\cdot \no^h  &=  H^h \no^h   & & \text { on } \II(0) \times(0, T),\label{transformed stress jump}\\
\jump{ \tilde{v } } &=0 & & \text { on } \II(0) \times(0, T), \\
\left.\tilde{v }\right|_{\partial \Omega} &=0 & & \text { on } \partial \Omega \times(0, T), \\
\left.\tilde{v }\right|_{t=0} &=v_0 & & \text { in } \Omega, 
\end{align}
\end{subequations}
as well as
\begin{subequations}\label{flat system2}
\begin{align}
\partial_{t} h  &= H^h+\tilde{v }\cdot\no^h& &\text { on } \II(0)\times (0,T),\\
 h|_{t=0} &=0 & &\text { on } \II(0).
\end{align}
\end{subequations}
To state the regularity of solutions to the above system,  
we finally introduce the following function spaces:
 \begin{subequations}\label{local well functional space}
\begin{align}\label{def Xq}
&  \mathbb{V}_p:=  W_p^{1}\left( J; L^p(\Omega)\right)  \cap L^p\left( J; W_p^{2} (\mathring{\O}(0) ) \cap W_p^{1}(\Omega)\right),  \\
&  \mathbb{P}_p:= \Big\{L^p\left( J; \mathring{W}_p^{1} (\mathring{\O}(0) )\right)\Big\vert \int_\Omega \tilde{\pi}(\cdot,t)=0,\nonumber\\  &\qquad\qquad\qquad\qquad ~\jump{ \tilde{\pi} } \in W_p^{1-1/p, \frac{1}{2}\left(1-1/p\right)}(\II(0){\times} (0,T))\Big\} ,\label{def Yq}\\
& \mathbb{H}_p:=  W_{p}^{1}\left( J; W_{p}^{1-1/p}(\II(0))\right)\cap L^p\left( J; W_{p}^{3-1/p}(\II(0))\right).\label{space of h}
\end{align}
\end{subequations}

 \begin{lemma}
The following embeddings hold true:
\begin{subequations}
 \begin{align}
 &\mathbb{V}_p  \hookrightarrow C(J;W_p^{2-  2/p}(\mathring{\O}(0))),\label{interpolation conti tX}\\
  &\mathbb{H}_p\hookrightarrow C\( J ; W_{p}^{3-3/p}(\II(0))\)\hookrightarrow C ( J ; C^2(\II(0))).\label{interpolation conti tZ}
 \end{align}
\end{subequations}
\end{lemma}
\begin{proof}
To compute the interpolation of Sobolev--Slobodeckij  spaces  we recall
\begin{align}
\(W_{p}^{m_1}(U),W_{p}^{m_2}(U)\)_{\theta,p} =W_p^s(U),\qquad s=(1-\theta)m_1+\theta m_2,\label{interpolation general}
\end{align}
where $\theta\in (0,1)$ and $U$ is an open set or a closed compact manifold.
We deduce from \eqref{interpolation general} that 
\begin{align}
\left(W_{p}^{1-1/p}(U), W_{p}^{3-1/p}(U)
\right)_{1-1/p, p}&=W_{p}^{3-\frac{3}{p}}(U) \hookrightarrow C^{2}(U),\label{embedding Iv conti}\\
\left(L_{p}(U), W_{p}^{2}(U)
\right)_{1-1/p, p}&=W_{p}^{2-\frac{2}{p}}(U).
\end{align}
These combined with    the   interpolation inequality \eqref{interpo in conti time} leads to \eqref{interpolation conti tX} and the first embedding in \eqref{interpolation conti tZ}. Concerning the second embedding in \eqref{interpolation conti tZ}, as  $p>d+2$  and   $\II(0)$ is $(d{-}1)$-dimensional,  we have  $3-\frac{3}{p}-\frac{d-1}{p}>2$. Then the result follows from  Morrey's  embedding.
 \end{proof}

The following result is a simpler  version of Abels and Moser~\cite[Theorem 4.1]{zbMATH06951007}.
   \begin{proposition}
   Let the assumptions of Proposition \ref{local strong limit sys} be in place.
  Let $p>d+2$ and $2<q<3$ with $1+\frac{d+2}{p}>\frac{d+2}{q}$.  
 Then  there exists $T>0$ such that      the system consisting of  
\eqref{flat system1}  and  \eqref{flat system2} has  a unique solution with  
\begin{align}
(\tilde{v },\tilde{\pi},h)\in \mathbb{V}_q\times \mathbb{P}_q\times \mathbb{H}_p.\label{unbalanced regularity}
\end{align}

 \end{proposition}
 We make a comment on the   regularity of $(\tilde{v },\tilde{\pi})$.
\begin{remark}
According to the regularity of $\tilde{v}$ by \eqref{def Xq} and the embedding \eqref{interpolation conti tX}, we must have $v _0(\cdot)=\tilde{v}( \cdot,0)\in W_q^{2- 2/q}(\mathring{\O}(0))$. If $q> 3$, then the trace estimate implies $\nabla v_0\in W_q^{1- 2/q}(\mathring{\O}(0))\hookrightarrow W_q^{1- 3/q}(\II(0))$. In order to guarantee that the restrictions  from outer and inner domains $\O^\pm(0)$ give the same trace,  an additional compatibility condition must be added. Such a condition is caused by the possible jump of $\nabla v_0$ across $\II(0)$.
To avoid such a compatibility condition, Abels and Moser~\cite{zbMATH06951007} assume $q<3$, and this causes unbalanced regularities \eqref{unbalanced regularity}.
 However, in the current system \eqref{flat system1}, we work 
in the regime of coinciding shear viscosities of the two phases. 
It turns out that this assumption necessitates continuity of the flow gradient
across the interface, see Lemma~\ref{lem:noJump} below.
In particular, we can simply assume $v_0\in W_q^{2- 2/q}(\O)$. 
\end{remark}

\begin{lemma}
\label{lem:noJump}
Under the assumption of coinciding shear viscosities $\mu_+=\mu_-$
of the two fluid phases, it follows that
the flow gradient $\nabla v$ is continuous across $\II$, i.e., $\jump{\nabla v}=0$. 
As a result, we can improve \eqref{regularity of u} to 
 \begin{align}
 v \in   W_{p}^{1}\left( J; L^{p}(\Omega)\right)  \cap L^{p}\left( J; W_{p}^{2}\left(\Omega \right)  \right). \label{regularity of u across}
 \end{align}
 \end{lemma}

\begin{proof}
In the proof, we shall abbreviate $\no_{\II}$ by $\no$.
 It follows from \eqref{regularity of u} that 
\begin{equation}
\nabla v\in W^1_p(\mathring{\O}(t))~ ~ \text{for a.e.\ } t\in[0,T],\label{nabla v w1p}
\end{equation}
and thus $\jump{ \nabla v}$ makes sense. By   \eqref{jump u}, we deduce that 
the tangential derivative of $v$ do not  jump across $\II$, i.e.,
\begin{align}
\label{eq:noTangentialJump}
\jump{ \nabla^{\tan}  v_i}=0~\text{ on }\II~\text{ for }~1\leq i\leq d,
\end{align}
   where $\nabla^{\tan} v_i:= (\mathrm{Id}-\no\otimes \no) \nabla v_i$ is the tangential gradient of $v_i$.
	Thus,  one can verify that   $\nabla\cdot v=0$ in  $\O$  in the sense of distribution 
	(see \cite[Lemma 2.5]{Alberti2003} for an even more general situation), i.e., 
	for any $\varphi\in C_c^1(\O)$ it holds  $\int_\O  v\cdot \nabla \varphi\, dx=0$. 
	This in turn implies by the regularity of~$v$ that $\nabla\cdot v = 0$ a.e.\ in~$\Omega$,
	and therefore by~\eqref{eq:noTangentialJump}
	\begin{align}
	\label{eq:noNormalNormalJump}
	\jump{\no \otimes \no \colon \nabla v}
	= \jump{(\no \otimes \no-\mathrm{Id}) \colon \nabla v} + \jump{\nabla \cdot v} = 0.
	\end{align}
	Next, multiplying \eqref{jump stress} by any vector field~$\mathbf{t}$ orthogonal to~$\no$, we deduce 
	from~\eqref{eq:noTangentialJump}
\begin{align} 
\label{eq:noTangentialNormalJump}
0 = \jump{ 2 \nabla^\mathrm{sym} v: \no \otimes \mathbf{t}}
= \jump{\nabla v : \mathbf{t} \otimes \no}.
\end{align}
Hence, the claim of Lemma~\ref{lem:noJump} follows from~\eqref{eq:noTangentialJump},
\eqref{eq:noNormalNormalJump} and~\eqref{eq:noTangentialNormalJump}.
	%
%
\end{proof}

\subsection{Proof of Proposition~\ref{local strong limit sys}.}
The regularities of $(\tilde{v},\tilde{\pi})$ and $\tilde{h}$  in \eqref{unbalanced regularity}  are not balanced. 
  It remains to show further  integrability   of $(\tilde{v},\tilde{\pi})$ by solving \eqref{flat system1} separately under  some additional compatibility conditions. A complete proof of this   will be  
	quite lengthy and technical, and we thus only give a sketch of the proof.
 To this end, we write   the system \eqref{flat system1} in an abstract form  
 \begin{align}\label{operator equ}
 \mathscr{L}\begin{pmatrix}
 \tilde{v} \\ \tilde{\pi}  
 \end{pmatrix}= \mathscr{N}\begin{pmatrix}
 \tilde{v} \\ \tilde{\pi}  
 \end{pmatrix}
 \end{align}

 where 
 \begin{align}\label{linear operator L}
 \mathscr{L}\begin{pmatrix}
 \tilde{v} \\ \tilde{\pi}  
 \end{pmatrix}
:=\left(\begin{array}{c}
\partial_{t} \tilde{v }- \Delta \tilde{v }-\nabla \tilde{\pi} \\
\div  \tilde{v } \\
\jump{ 2 D \tilde{v } -\tilde{\pi} \mathrm{Id} } \no_{\II} \\
\tilde{v}|_{t=0}
\end{array}\right)
\end{align}
is the linearized operator
and 
$\mathscr{N}$ is the nonlinear one:
\begin{align}\label{local well N operator}
 \mathscr{N}\begin{pmatrix}
 \tilde{v }\\ \tilde{\pi}  
 \end{pmatrix}:=\left(\begin{array}{c}
 \mathbf{a} \left(\tilde{v },\tilde{\pi},h\right) -\tilde{v } \cdot \nabla^h \tilde{v }+\nabla^h \tilde{v } \,\partial_{t} \tilde{\Theta}_{h} \\
(\div  \tilde{v }-\div ^h \tilde{v })-  \fint_{\Omega} (\div  \tilde{v }-\div^h \tilde{v }) d x \\
\mathbf{b}\left(\tilde{v },\tilde{\pi},h\right)+  H^h \no^h \\
v_0
\end{array}\right).
 \end{align}
In \eqref{local well N operator}, we used the definitions from~\eqref{pertubations} and 
\begin{subequations}\label{various differences}
\begin{align}
\mathbf{a} (\tilde{v }, \tilde{\pi},h) &:=( \div ^h \nabla^h \tilde{v }-\Delta \tilde{v })+(\nabla^h-\nabla) \tilde{\pi}, \label{various differences1}\\
\mathbf{b}(\tilde{v }, \tilde{\pi},h) &:=\jump{  2 D \tilde{v } -\tilde{\pi} ~\mathrm{Id} }(\no_{\II}-\no^h)+2\jump{ D \tilde{v }-D^h \tilde{v }) } \no^h. \label{various difference2}
\end{align}
\end{subequations}
\begin{remark}
It is easy to verify that  the operator equation~\eqref{operator equ} is equivalent to~\eqref{flat system1}
 except that  the second equation in the latter, i.e., $\div ^h \tilde{v } =0$,  is replaced by
\begin{align}\label{div integral}
\div^h  \tilde{v } =   \fint_{\Omega}  \div^h \tilde{v }\, d x
\end{align}
after simplification.
It is obvious that the former  equation implies the latter one. The opposite direction is proved in    \cite[p.\ 51]{MR3062573}:
\begin{align*}
&\fint_{\Omega}  \div^h \tilde{v }\, d x\int_\O  |\det \nabla \tilde{\Theta}(x,t)| \, dx 
\\&\overset{\eqref{div integral}}
=\int_\O \div^h  \tilde{v } |\det \nabla\tilde{\Theta}(x,t)| \, dx
\\&\overset{\eqref{gradient equ}}
= \int_\O (\div_y v)|_{y=\tilde{\Theta}_{h}(x, t)} |\det \nabla\tilde{\Theta}(x,t)| \, dx
\\
&~~=\int_\O \div_y v(y)\, dy=0.
\end{align*}
\end{remark}

Recall  from  \eqref{flat system2} that $h$ starts from $0$. By  the regularity 
$h\in \mathbb{H}_p$, cf.\ \eqref{space of h},  it can be shown that,   within a short time period  $[0,T]$, the nonlinear operator $\mathscr{N}$ defined by~\eqref{local well N operator} is locally Lipschitz  continuous between a pair of Banach  spaces. Moreover, we shall   show that 
$\mathscr{L}$  is an isomorphism between this pair. These two results together with Banach's fixed point theorem then lead to the proof of Proposition~\ref{local strong limit sys}.

The invertibility  of the linear operator $\mathscr{L}$ from~\eqref{linear operator L}
 corresponds  to  the solvability of the  following linear system:
\begin{subequations}\label{local well linear operator}
 \begin{align}
\partial_{t} \tilde{v }-  \Delta \tilde{v }+\nabla \tilde{\pi} &=f & & \text { in } \mathring{\O}(0) \times(0, T), \label{local well linear operator1}\\
\div  \tilde{v } &=g  & & \text { in } \mathring{\O}(0) \times(0, T), \label{local well linear operator2}\\
\jump{ 2 D \tilde{v } -\tilde{\pi}~ \mathrm{Id} } ~ \no_{\II} &=w & & \text { on } \II(0) \times(0, T),\label{local well linear operator3}\\
\jump{ \tilde{v } } &=0 & & \text { on } \II(0) \times(0, T),\label{local well linear operator4} \\
\tilde{v }|_{\partial \Omega} &=0 & & \text { on } \partial \Omega \times(0, T), \label{local well linear operator5}\\
\tilde{v }|_{t=0} &=\tilde{v }_0 & & \text { in } \Omega.
\end{align}
\end{subequations}
Here, we assume
   $(f,g,w)\in L^p_{x,t}\times \mathbb{G}_p\times \mathbb{W}_p$  where   
\begin{subequations}\label{local well compatibility}
\begin{align}
&\mathbb{G}_p:=\left\{g\in L^p\(0,T;W^1_p(\mathring{\O}(0))\)\Big\vert\int_\O g(t,\cdot)=0
\text{ for a.e.\ } t\in (0,T)\right\},\label{local well integral constraint}\\
 &  \mathbb{W}_p:=W^{\frac 12(1- \frac 1p)}_p\Big(0,T;L^p(\II(0) )\Big)\cap L^p\Big(0,T;W^{1-  \frac 1p}_p(\II(0) )\Big),\label{wp space}
\end{align}
\end{subequations}
and we assume that $(g,w,\tilde{v}_0)$ satisfy the following compatibility conditions:
\begin{subequations}\label{assume v0}
\begin{align}
\tilde{v} _0 &\in W_p^{2-2/p}(\mathring{\O}(0)),\\
  \div  \tilde{v}_0 &= g|_{t=0} &&\text{in}~\mathring{\O}(0)\label{local well initial constraint}\\
 \tilde{v}_0|_{\partial\O} &= 0,\\
  \jump{ \tilde{v}_0 } &= 0 &&\text{on}~\II(0),\\
 \Pi_{\II}\(\jump{ 2D \tilde{v}_0 }\cdot \no_{\II}\)\Big |_{t=0}
&= \Pi_{\II} w\Big |_{t=0}.\label{local well jump stress}
 \end{align}
\end{subequations}

\begin{remark}
We note that the integral constraint  in \eqref{local well integral constraint} follows by integrating \eqref{local well linear operator2} and using \eqref{local well linear operator4} and \eqref{local well linear operator5}. The compatibility condition \eqref{local well initial constraint} follows by taking $t=0$ in  \eqref{local well linear operator2}. Finally, 
\eqref{local well jump stress} is a consequence of taking the tangential 
projection of~\eqref{local well linear operator3} and then restricting it to $t=0$.
\end{remark}

The solvability of the system \eqref{local well linear operator}  is due to   the following maximal regularity result   in \cite{MR2548875}.

\begin{lemma}
Assuming the compatibility conditions \eqref{assume v0}, the system  \eqref{local well linear operator} has a unique solution with the following estimate under notations in \eqref{local well functional space}:
\begin{align}
\|(\tilde{v} ,\tilde{\pi})\|_{\mathbb{V}_p\times \mathbb{P}_p}\leq C\(\|f\|_{L^p_{x,t}} +\|(g,w)\|_{  \mathbb{G}_p\times \mathbb{W}_p}+\|\tilde{v}_0\|_{W_p^{2-2/p}(\mathring{\O}(0))}\)
\end{align}
where $C$ is a constant depending on  the geometry of $\mathring{\O}(0)$.
\end{lemma}
 Note that the resolvent estimate for \eqref{local well linear operator} 
 in bent half-spaces is given in \cite[Theorem 6.1]{MR1978384}, which implies the  maximal regularity of the instationary system when  $p=2$ (see   \cite{MR3062573} for a short proof). 
 However, the $L^p$ maximal regularity of the instationary system does not follow directly from the corresponding resolvent estimate.

To show  the contraction property of  \eqref{local well N operator}, 
we need to study the regularity of the mapping $(\nabla \tilde{\Theta}_{h})^{-1} - \mathrm{Id}$.
  
\begin{lemma}
Under the regularity assumption $h\in \mathbb{H}_p$, cf.\ \eqref{space of h}, we have 
 \begin{align}\label{est hanzawa}
&\left\| (\nabla \tilde{\Theta}_{h})^{-1}  -\mathrm{Id}\right\|_{C([0,T],W^{2-\frac 2 p}_p(\O))}+\left\| (\nabla \tilde{\Theta}_{h})^{-1}-\mathrm{Id}  \right\|_{W^{\frac 12}_p (0,T; L^p(\O) )}\xrightarrow{T\to 0}0.
\end{align}
\end{lemma}

\begin{proof}
Recalling \eqref{space of h} and   \eqref{E extension 123}, we have 
\begin{align}
Eh \in W_{p}^{1}\left( 0,T; W_{p}^{1}\big(B_{\delta}(\II(0))\big)\right)
\cap L^p\left( 0,T; W_{p}^{3}\big(B_{\delta}(\II(0))\big)\right). \label{regularity Eh}
\end{align}
By \eqref{extend hanzawa}, we obtain the same regularity for $\tilde{\Theta}$. So using    \eqref{general interpolation} and \eqref{interpo in conti time} yields 
\begin{align}\label{interpolation reg theta}
\tilde{\Theta}\in  C( [0,T] ; W_{p}^{3-\frac 2p}(\O)) \cap W^{\frac 12}_p (0,T; W^2_p(\O) ).
\end{align}
 As  $Eh|_{t=0}=0$,  for a sufficiently small $T$, we deduce from \eqref{extend hanzawa} and \eqref{interpolation reg theta} that   
\begin{align}
\sup_{x\in \O,t\in [0,T]}|\nabla \tilde{\Theta}_{h} - \mathrm{Id}|\leq 1/2.
\end{align}
By a Taylor expansion, we have   
\begin{align}
(\nabla \tilde{\Theta}_{h})^{-1}-\mathrm{Id}=\mathfrak{F}(\nabla \tilde{\Theta}_{h}-\mathrm{Id})
\end{align}
 where $\mathfrak{F}(B)=\sum_{k=1}^\infty (-1)^k B^k$ is a smooth matrix-valued function for $|B|<1$. This combined with  \eqref{interpolation reg theta}  and \eqref{extend hanzawa} yields  
 \begin{align*} 
&\left\| (\nabla \tilde{\Theta}_{h})^{-1}  -\mathrm{Id}\right\|_{C([0,T],W^{2-2/p}_p(\O))}
+\left\| (\nabla \tilde{\Theta}_{h})^{-1}-\mathrm{Id}  \right\|_{W^{1/2}_p(0,T; L^p(\O))}\\
~~~&\leq  \left\|  Eh\right\|_{C([0,T],W^{3-2/p}_p(B_\delta(\II(0)))}
+\left\| Eh  \right\|_{W^{1/2}_p(0,T; W^1_p(B_\delta(\II(0))))}.
\end{align*}
The first term on the right hand side vanishes as $T\downarrow 0$ because of \eqref{interpolation reg theta} and  $Eh|_{t=0}=0$. Concerning the second one, we have for $s'=\frac12 +\epsilon$, $0<\epsilon \ll 0$ and $p> d+2$  that 
\begin{align*}
\left\| Eh  \right\|_{W^{\frac 12}_p(0,T; W^1_p )}& \overset{\eqref{wsp embed cs}} \leq C_{s',p} T^{s'-\frac 12+\frac 1p} \left\| Eh  \right\|_{C^{s'}([0,T]; W^1_p)}\\
& \leq C_{s',p} T^{s'-\frac 12+\frac 1p} \left\| Eh  \right\|_{W^{s'+\epsilon+1/p}_p([0,T]; W^1_p)}\\
&\overset{\eqref{regularity Eh}}\leq C_{s',p} T^{s'-\frac 12+\frac 1p} C_{Eh}.\end{align*}
With this, we may conclude with our sketch of the proof.
 \end{proof}
 
 \begin{lemma}
Under the regularity assumption $h\in \mathbb{H}_p$, cf.\ \eqref{space of h},   
the mapping \eqref{local well N operator}
\begin{align}\label{scrN}
\mathscr{N}:\mathbb{V}_p\times \mathbb{P}_p\mapsto L^p_{x,t}\times \mathbb{G}_p\times \mathbb{W}_p\times W_p^{2-2/p}(\mathring{\Omega}(0))=:\mathbb{N}_p
\end{align}
 is locally Lipschitz. Moreover,  for 
 \begin{align}\label{vi in BR}
 (v_i,\pi_i)\in B_R\(\mathbb{V}_p\times \mathbb{P}_p\)\text{ with }1\leq i\leq 2,
 \end{align}
  we have the following estimate:
    \begin{align}\label{est diff}
 &\left\|\mathscr{N}\begin{pmatrix}
 v_1\\ \pi_1  
 \end{pmatrix}-\mathscr{N}\begin{pmatrix}
v_2\\ \pi_2
 \end{pmatrix} \right\|_{  \mathbb{N}_p}\leq  C(R,T)  \|(v_1-v_2,\pi_1-\pi_2)\|_{\mathbb{V}_p\times \mathbb{P}_p},
 \end{align}
 where $\lim_{T\downarrow 0}C(R,T)=0$ for each fixed $R>0$.
   \end{lemma}     
\begin{proof}
A full proof of \eqref{est diff} will be quite lengthy and technical, 
so  we again only  sketch a few key steps here. 

We first deduce from \eqref{regularity Eh} and $Eh|_{t=0}=0$ that 
 \begin{align}
\mathfrak{h}(\tau):=\| Eh\|_{C([0, \tau] ; W_{p}^{3- 2/p}(B_{\delta}(\II(0))))}\label{GT def}
\end{align}
 is  continuous on $[0,T)$ and that $\lim_{\tau\downarrow 0} \mathfrak{h}(\tau)=0$.
To verify the  Lipschitz continuity~\eqref{est diff},  we start from  the only  
nonlinear term $\tilde{v } \cdot \nabla^h \tilde{v }$ of it. Let $\bar{v}=v_1-v_2$. Then
\begin{align*}
&\left\| v_1 \cdot \nabla^h v_1 -v_2 \cdot \nabla^h v_2 \right\|_{L^p_{x,t}}\\
&=\left\| \bar{v} \cdot \nabla^h v_1 -v_2 \cdot \nabla^h \bar{v} \right\|_{L^p_{x,t}}\\
&\leq     C_1(R)  \(1+\mathfrak{h}(T)\) T^{\frac 1 p} \|\bar{v}\|_{C([0,T], W^1_p)}\qquad \text{by } \eqref{vi in BR}\text{ and } \eqref{est hanzawa}\\
&\leq     C_2(R)   T^{\frac 1 p} \|\bar{v}\|_{\mathbb{V}_p }.
\end{align*}
All the remaining terms defining \eqref{local well N operator} are linear. 
We shall merely   estimate $\jump{ D \tilde{v }-D^h \tilde{v }) } \no^h $ in \eqref{various difference2}, 
since all the other terms can be treated  in a similar way. 
Note that   $W^1_p(\O)$   is a   Banach algebra for $p>d+2$. Then
\begin{align*}
&\left\| \jump{ D v_1-D^h v_1) }  -\jump{ D v_2-D^h v_2) } \right\|_{\mathbb{W}_p}\\
&~~~\overset{\eqref{def D h}}
\lesssim   \left\| ((\nabla \tilde{\Theta}_{h})^{-1} - \mathrm{Id})
\jump{\nabla \bar{v}}\right\|_{\mathbb{W}_p}\\
&~~~\overset{\eqref{wp space}}\lesssim    \left\| ((\nabla \tilde{\Theta}_{h})^{-1} - \mathrm{Id}) \nabla \bar{v}\right\|_{L^p(0,T;W^1_p(\mathring{\O}(0))) \cap W^{1/2}_p(0,T;L^p( \O  ))}
\\
 &~~~~~\leq   \|(\nabla \tilde{\Theta}_{h})^{-1} - \mathrm{Id} \|_{C([0,T],W^{2-2/p}_p(\O))}
\left\|   \bar{v}\right\|_{L^p(0,T;W^2_p(\mathring{\O}(0)))\cap W^{1/2}_p(0,T;W^1_p(\mathring{\O}(0)))}
\\
&~~~~~~~~~ +\left\|  (\nabla \tilde{\Theta}_{h})^{-1} - \mathrm{Id} 
\right\|_{  W^{1/2}_p(0,T;L^p(\O(0)))}\| \nabla \bar{v}\|_{L^\infty_{x,t}}\\
&\overset{\eqref{est hanzawa},\eqref{interpolation conti tX}} \leq   
C(T)~\| \bar{v}\|_{W^{1}_p(0,T;L^p(\mathring{\O}(0)))\cap L^p(0,T;W^2_p(\mathring{\O}(0)))}
\end{align*} 
with  $C(T)\xrightarrow{T\to 0} 0$. 
\end{proof}

\begin{proof}[Proof of Proposition~\ref{local strong limit sys}]
 Combining the above two lemmas, we deduce 
 $(\tilde{v},\tilde{\pi})\in \mathbb{V}_p\times  \mathbb{P}_p$ (cf.\ \eqref{local well functional space}) 
by a fixed point argument.   
To obtain further regularity of $h$, we recall  from \eqref{def Xq} that 
\begin{equation}
\tilde{v}\in \mathbb{V}_p\hookrightarrow W_p^{1-\frac 1{2p}}\left( J; L^p(\II(0))\right)  
\cap L^p\big( J; W_p^{2-\frac 1p} (\II(0) )  \big).
\end{equation}
Hence, by solving the quasilinear parabolic equation \eqref{flat system2} can we  
improves the regularity from of $h\in \mathbb{H}_p$, cf. \eqref{space of h}, to \eqref{regularity h}.
 This in turn is sufficient to conclude with the regularity
stated in Definition~\ref{def:strongSolSharpInterfaceLimit}.
 \end{proof}

\section*{Acknowledgements}
S.\ Hensel has received funding from the Deutsche Forschungsgemeinschaft (DFG, German Research Foundation) 
under Germany's Excellence Strategy -- EXC-2047/1 -- 390685813. Y.\ Liu is partially 
supported by NSF of China under Grant  11971314.




\end{document}